\documentclass[a4paper,11pt]{amsart}

\usepackage[utf8]{inputenc}

\usepackage[a4paper,right=2.75cm,left=2.75cm,bottom=2.75cm, top=2.75cm]{geometry}

\usepackage{lmodern}
\usepackage{latexsym}
\usepackage{graphics}
\usepackage{amsmath}
\usepackage{amssymb}
\usepackage{bm}
\usepackage{mathrsfs}
\usepackage{enumerate,enumitem}
\usepackage{csquotes}
\usepackage{subcaption}
\usepackage{pgfplots,tikz}
\usetikzlibrary{decorations.pathreplacing}
\pgfplotsset{compat=1.10}
\usepgfplotslibrary{fillbetween}
\usetikzlibrary{patterns}
\usetikzlibrary{intersections, calc, angles}
\usepackage{caption}  
\usepackage{enumerate}
\usepackage{comment}

\usepackage{xfrac}
\usepackage{comment}

\usepackage{hyperref}
\usepackage[capitalise,noabbrev]{cleveref}

\hypersetup{colorlinks=true, 
	linkcolor=blue,
	citecolor=red,
}
\usepackage[showonlyrefs]{mathtools}

\usepackage{etoolbox,kvoptions,logreq,pdftexcmds}

\definecolor{racing}{rgb}{0.7,0.1,0.2}
\definecolor{french}{rgb}{0,0.2,0.7}

\numberwithin{equation}{section} 

\newtheorem{theorem}{Theorem}[section]
\newtheorem{corollary}[theorem]{Corollary}

\newtheorem{problem}[theorem]{Problem}
\newtheorem{lemma}[theorem]{Lemma}
\newtheorem{proposition}[theorem]{Proposition}
\theoremstyle{definition} 
\newtheorem{definition}[theorem]{Definition}

\newtheorem{remark}[theorem]{Remark}

\newcommand{\R}{\mathbb{R}}	% Real numbers
\newcommand{\N}{\mathbb{N}} % Natural numbers
\newcommand{\dx}{\,\mathrm{d}x}	% dx
\newcommand{\ds}{\,\mathrm{d}S}	% ds
\renewcommand{\d}{\mathrm{d}}
\newcommand{\weak}{\rightharpoonup}
\newcommand{\nnu}{\bm{\nu}}  % nu exterior normal

\newcommand{\eps}{\varepsilon}

\newcommand{\norm}[1]{\left\lVert #1 \right\lVert}
\newcommand{\abs}[1]{\left| #1 \right|}
\newcommand{\sub}{\subseteq}

\DeclareMathOperator{\dist}{\mathrm{dist}}

\newenvironment{bvp}{\left\{\begin{aligned}  }{\end{aligned}\right.}

\usepackage{amsaddr}
\author[R. Ognibene]{Roberto Ognibene}
\address{Dipartimento di Matematica e Applicazioni \\
	Universit\`a degli studi di Milano-Bicocca \\
	Via Roberto Cozzi, 55, 20126 Milano, Italy}
\email{roberto.ognibene@unimib.it}

\author[B. Velichkov]{Bozhidar Velichkov}
\address{Dipartimento di Matematica \\
	Universit\`a di Pisa \\
	Largo Bruno Pontecorvo, 5, 56127 Pisa, Italy}
\email{bozhidar.velichkov@unipi.it}

\begin{document}
	
	\title[A survey on the optimal partition problem]{A survey on the optimal partition problem}

	\keywords{optimal partition, free boundary, harmonic maps into singular spaces, multivalued harmonic functions}
	\subjclass{
		35R35,   % Free boundary problems for PDEs
		49Q10   % Optimization of shapes other than minimal surfaces
	}

	\begin{abstract}
		This survey synthesizes the current state of the art on the regularity theory for solutions to the optimal partition problem. Namely, we consider non-negative, vector-valued Sobolev functions whose components have mutually disjoint support, and which are either local minimizers of the Dirichlet energy or, more generally, critical points satisfying a system of variational inequalities. This is particularly meaningful as the problem has emerged on several occasions and in diverse contexts: our aim is then to provide a coherent point of view and an up-to-date account of the progress concerning regularity of the solutions and their free boundaries, both in the interior and up to a fixed boundary.
		
	\end{abstract}
	
	\maketitle
	
	% \tableofcontents

	\section{Introduction}\label{sec:intr}
	
	Let $D\sub\R^d$ be an open and connected set, let $N\in\N$ and let
	\begin{equation*}
		\Sigma_N:=\{X\in\R^N\colon X_i\ge 0\ \text{and}\  X_i X_j=0~\text{for all }i,j=1,\dots,N,~i\neq j\}.
	\end{equation*}
	We work in the space
	\begin{equation*}
		H^1(D,\Sigma_N):=\{u\in H^1(D,\R^N)\colon u(x)\in\Sigma_N~\text{for a.e. }x\in D\},
	\end{equation*}
	pointing out that the fact that $u(x)\in\Sigma_N$ for a.e. $x\in D$ means that all the components 
	$$u_1,\dots,u_N\colon D  \to \R$$
	are non-negative functions with disjoint support (we might also say that the functions are \emph{segregated}). The main character of the present survey are non-negative local minimizers (and, to some extent, a certain notion of critical points, see \Cref{def:extremality}) of the corresponding energy functional
	\begin{equation*}
		E(u,D):=\int_D |\nabla u|^2\dx=\sum_{i=1}^N \int_D|\nabla u_i|^2\dx.
	\end{equation*}
	\begin{definition}\label{def:minimizer}
		Given a function $u:D\to\Sigma_N$, we say that 
		\begin{itemize}
			\item $u$ is a minimizer of $E$ in $D$ if $u\in H^1(D,\Sigma_N)$ and it holds
			\begin{equation*}
				E(u,D)\leq E(v,D)\quad\text{for all }v\in H^1(D,\Sigma_N)~\text{such that }u-v\in H^1_0(D,\R^N).
			\end{equation*}
			\item $u$ is a local minimizer of $E$ in $D$ if:
			\begin{itemize}
				\item $u\in H^1_{\textup{loc}}(D,\Sigma_N)$, that is, $u\in H^1(\Omega,\Sigma_N)$ for all open $\Omega\Subset D$; 
				\item $u$ is a minimizer of $E$ in $\Omega$.
			\end{itemize}   
		\end{itemize}
		%  made in the class of admissible functions
		%  \begin{equation*}
			%      \mathcal{A}(D,N):=\{u\in H^1_{\textup{loc}}(D,\Sigma_N)\colon u_i\geq 0~\text{in }D\},
			%   \end{equation*}
		%   where for $u:D\to\Sigma_N$ we say that $u\in H^1_{\textup{loc}}(D,\Sigma_N)$ if $u\in H^1(\Omega,\Sigma_N)$ for all $\Omega\Subset D$.
		% Namely, we give the following.
		%  \begin{definition}\label{def:minimizer}
			%    We say that $u\in \mathcal{A}(D,N)$ is a \emph{local minimizer} (or a \emph{variational solution} of the optimal partition problem) if for every $\Omega\Subset D$ there holds
			%   \begin{equation*}
				%       E(u,\Omega)\leq E(v,\Omega)\quad\text{for all }v\in H^1(\Omega,\Sigma_N)~\text{such that }u-v\in H^1_0(\Omega,\R^N).
				%   \end{equation*}
			% We say that $u$	is a \emph{global minimizer} (or a \emph{global variational solution}) if $u\in\mathcal A(\R^D,N)$ is a variational solution in $\R^d$. 
		\end{definition}
		
		The focus of this survey is on the regularity of the local minimizers (and critical points): besides the regularity of the functions themselves (it turns out that they are Lipschitz continuous, see \Cref{sec:known_results}), we are mainly interested in the regularity of the free boundary (the nodal set), defined as
		\begin{equation*}
			\mathcal{F}_D(u):=\bigcap_{i=1}^N \{x\in D\colon u_i(x)= 0\}.
		\end{equation*}

		It turns out that the problem we just described has appeared in the literature at different times and in diverse contexts as a model for various types of phenomena. 
		In \Cref{sec:pov} we describe the ones we are aware of, that is: harmonic maps with singular target spaces (\Cref{subsec:harmonic}), limit of competition-diffusion models (\Cref{subsec:comp-diff}), spectral optimal partition problems (\Cref{subsec:optimal}) and multi-valued harmonic functions (\Cref{subsec:2-val}).
		
		We wish to emphasize a crucial aspect: since the model we focus on (i.e.,  minimizers of $E(\cdot,D)$ acting on $H^1(D,\Sigma_N)$) represents a particular case within each of the broader theories discussed (we refer to each subsection for instances of such generalizations), its study plays a central role. On one hand, this focus allows for a clearer exposition, more straightforward statements, and a reduction in the technicalities of the proofs. On the other hand, and more importantly, its thorough understanding is crucial: this model, in fact, acts as a bridge between different theories, so that any further step made in its study is expected to translate into progress across all the areas related to it.
		
		This being said, in \Cref{sec:known_results}, after introducing the main tools used to study such problems, we review the known results concerning regularity of the free boundary, both in the interior (\Cref{sec:interior}) and up to the fixed boundary (\Cref{sec:boundary}).
		
		% Finally, in \Cref{sec:open_questions} we state some open question in the field.
		
		\begin{remark}
			We emphasize that, since the optimal partition problem lies as the intersection of various fields, the known results about this problem that we recall in \Cref{sec:known_results} actually come as particular cases of more general results proved within a certain framework (among the ones described in \Cref{sec:pov}). Hence, since mentioning this fact any time would make the exposition heavier, we adopt the convention that this is generally the case and refer the reader to the cited papers for the complete statements.
		\end{remark}
		
		\begin{remark}
			This survey is addressed to researchers specialized in the analysis of PDEs and regularity theory. We presuppose the reader's familiarity with elliptic regularity, and the basic tools of free boundary problems (such as monotonicity formulas). This allows us to focus directly on the specific features of the optimal partition problem.
		\end{remark}
		
		\section{Points of view}\label{sec:pov}
		
		In this section, we describe the various theories for which our model, i.e. the functional $E(\cdot,D)$ acting on $H^1(D,\Sigma_N)$, represents a particular case.
		
		\subsection{Harmonic maps with singular target spaces}\label{subsec:harmonic}
		
		Harmonic maps are classically defined as critical points of the Dirichlet energy functional naturally associated to functions
		\begin{equation*}
			u\colon D\to\mathcal{N},
		\end{equation*}
		with $D\sub\mathcal{M}$ being an open set and $(\mathcal{M},g)$ and $(\mathcal{N},h)$ being smooth Riemannian manifolds. More precisely, by taking an extrinsic viewpoint, let us first assume that $(\mathcal{N},h)$ is isometrically embedded into $\R^N$, for some $N\in\N$ (this is always possible thanks to Nash's theorem \cite{nash}). This way, we can define the Dirichlet energy in a natural way, for $u\colon D\to \R^N$ as
		\begin{equation*}
			\int_{D}|\nabla_{g} u|^2\,\d V_{g}=\sum_{i=1}^N\int_{D}|\nabla_{g} u_i|^2\,\d V_{g},
		\end{equation*} 
		and the corresponding Sobolev space
		\begin{equation*}
			H^1(D,\mathcal{N}):=\{u\in H^1(D,\R^N)\colon u(x)\in \mathcal{N}~\text{for a.e. }x\in D\}.
		\end{equation*}
		We point out that one can also define Sobolev spaces of maps between manifolds (and, more in general, metric spaces) in an intrinsic way which does not make use of embedding theorems; we refer to the seminal paper \cite{korevaar-schoen-1} for the basics of this theory. Harmonic maps are then defined as critical points of the Dirichlet energy with respect to suitable outer variations (one needs to deal with the fact that the target space $\mathcal{N}$ is not necessarily linear); if a harmonic map is critical with respect to inner variations (which are always admissible), then it is usually called \emph{stationary}. However, in this paragraph, we consider a map $u\in H^1(D,\mathcal{N})$ being a minimizer of the Dirichlet energy, which is then always a stationary harmonic map in the sense above.
		
		The literature on harmonic maps is huge, even when restricted to the topic treated herein, i.e., regularity theory (we refer to \cite{giaquinta,Lin-Wang-book-harmonic-maps,Simon-book-harmonic-maps} for a detailed and  comprehensive study of regularity of harmonic maps). We here briefly recall just the main results in the field, trying to give an up-to-date account. Given an energy-minimizing map $u\in H^1(D,\mathcal{N})$, $D\sub\mathcal{M}$ open, for general smooth Riemannian manifolds $\mathcal{M}$ and $\mathcal{N}$, we know that there exists a closed set $\mathrm{Sing}\,(u)\sub D$ such that
		\begin{itemize}
			\item $u$ is $C^\infty$ in $D\setminus\mathrm{Sing}\,(u)$;
			\item $\mathrm{Sing}\,(u)$ is $(d-3)$-rectifiable and has locally finite $(d-3)$-Hausdorff measure.
		\end{itemize}
		These results are essentially sharp, the latter being quite recent (we refer to the original article \cite{naber-valtorta} for comments on its sharpness). We also remark such results have been proved in a series of paper and a cornerstone role is played by \cite{SU} by Schoen \& Uhlenbeck. In particular, when $\mathcal{M}$ and $\mathcal{N}$ are arbitrary smooth manifolds, one cannot exclude discontinuities in the minimizers, for instance one can think of the (nowadays) standard example 
		\begin{align*}
			\R^d\to \mathbb{S}^{d-1}, \qquad
			x\mapsto \frac{x}{|x|}\,,
		\end{align*}
		which is minimizing for $d\geq 3$ (see \cite{lin1987remarque}) and has a discontinuity at the origin. Nevertheless, while the structure and the geometry of the (smooth) domain manifold plays essentially no role concerning the regularity of the minimizers, if one requires additional structural assumptions on the target manifold, then stronger results can be obtained. In particular, if the target manifold has nonpositive sectional curvature, then the Dirichlet energy becomes convex and so harmonic maps (with fixed Dirichlet data) are unique and coincide with minimizers. Moreover, in such a setting, minimizers are smooth, essentially as a consequence of the validity of a Bochner-type formula. This last result is contained in the pioneering work by Eells \& Sampson \cite{ES}. However, for a nice exposition on the topic, we refer to \cite{schoen}.
		
		One can then wonder how the regularity properties of the minimizers depends on the regularity of the target manifold $\mathcal{N}$. This question was essentially solved (up to some later refinements) in \cite{gromov-schoen}, where the authors consider harmonic maps with values into non-positively curved metric spaces (thus admitting singularities). Let us be more precise. Let $(X,d_X)$ be a length space, i.e. a metric space endowed with the intrinsic metric and such that every two points are connected by a distance-realizing curve, and let us assume it is \emph{non-positively curved in the sense of Alexandrov (NPC or CAT(0) space)}. The notion of curvature in the sense of Alexandrov is a generalization of the notion of sectional curvature adapted to general metric spaces. We give the following definition, referring to \cite{korevaar-schoen-1} for more details.
		\begin{definition}
			A metric space $(X,d_X)$ is \emph{non-positively curved} (NPC, or CAT(0) space) if 
			\begin{itemize}
				\item it is a geodesic space, that is: for every pair of points $P,Q\in X$ there exists a rectifiable curve whose length is $d_X(P,Q)$;
				\item the following \enquote{triangle comparison inequality}
				\begin{equation*}
					d_X(P,Q_t)^2\leq (1-t)d_X(P,Q)^2+td_X(P,R)^2-t(1-t)d_X(Q,R)^2,
				\end{equation*}
				holds for every triple of points $P,Q,R\in X$, with $Q_t\in X$ being the point on the geodesic segment joining $Q$ and $R$ such that $d_X(Q,Q_t)=t\,d_X(Q,R)$.
			\end{itemize}
		\end{definition}
		Given an open subset $D$ of a smooth manifold $(\mathcal M,g)$ and an NPC target space $(X,d_X)$, we consider minimizers for the Dirichlet energy associated to maps $u\colon D\to X$, i.e.
		\begin{equation*}
			\int_{D}|\nabla_g u|^2\,\d V_g.
		\end{equation*}
		To be precise, this quantity is classically defined in the case $X\sub \R^N$ (which is considered in \cite{gromov-schoen}), while in the case of a general metric space $X$ we refer to the intrinsic theory developed in \cite{korevaar-schoen-1} (see also \cite{hajlasz}). At this point, we observe the following: the fact that $\mathrm{Sing}(u)=\emptyset$ when the sectional curvature of the target manifold is non-positive does not depend on the smoothness of the manifold, but only on such geometric feature. Indeed, in \cite{gromov-schoen,korevaar-schoen-1} the authors proved the following.
		\begin{theorem}
			Let $D\sub\mathcal{M}$ be an open set, let $(X,d_X)$ be an NPC space and let $u\in H^1(D,X)$ be a minimizer of the Dirichlet energy. Then $u$ is locally Lipschitz continuous. In particular $\mathrm{Sing}(u)=\emptyset$.
		\end{theorem}
		Harmonic maps with singular NPC targets naturally emerge for various reasons, we name a few. First, they emerge as limits sequences of harmonic maps between smooth manifolds, see e.g. \cite{sinaei} and references therein. A remarkable example is the blow-down sequences. Namely, if $u\colon \R^d\to\mathcal{N}$, then  the blow-down rescaling $u_r(x):=r^{\gamma}u(x_0+x/r)$ (for $\gamma>0$ and $x_0\in D$) takes values in $r^\gamma\mathcal{N}:=\{r^\gamma x\colon x\in \mathcal{N}\}$. The space which emerges as a limit as $r\to 0^+$ of $r^\gamma\mathcal{N}$ (to be intended in the pointed Gromov-Hausdorff metric) might be singular. Moreover, if $\mathcal{N}$ has everywhere non-positive sectional curvatures, then the blow-down limit is NPC. Second, they might be a precious tool for studying rigidity questions for discrete groups, which is one of the main motivations for \cite{gromov-schoen} (see also \cite[Section 8]{hardt}). Lastly, and most importantly for our purposes, optimal partition problems. Let us choose
		\begin{equation*}
			X=\Sigma_N:=\{X\in\R^N\colon X_i\geq 0~\text{and } X_i X_j=0~\text{for all }i,j=1,\dots,N,~i\neq j\},
		\end{equation*}
		i.e. we assume $X$ to be the union of the nonnegative part of the coordinate axes of $\R^N$, then we precisely recover the optimal partition problem introduced in \Cref{sec:intr}, with intrinsic distance
		\begin{equation}\label{eq:d_Sigma}
			\mathrm{d}_{\Sigma_N}(X,Y)=\begin{cases}
				|X_i-Y_i|,&\text{if }X_j=Y_j=0~\text{for all }j\neq i, \\
				|X_i|+|Y_j|,&\text{if }X_i,Y_j\neq 0~\text{for some }i\neq j.
			\end{cases}
		\end{equation}
		The reason why we recalled this is that the space $\Sigma_N\sub\R^N$ is a simple example of non-positively curved space; thus, all the results proved in the setting of harmonic maps with values into NPC spaces holds in this case. The literature also in this field is large; without aiming at giving a complete list, besides \cite{gromov-schoen,korevaar-schoen-1} it is worth mentioning the independent work by Jost \cite{jost1}. We also include the lecture notes \cite{jost_lect} and some other contributions by Jost and coauthors \cite{jost2,jost3} (and references therein) and the various works by Mese and coauthors \cite{mese1,mese2,mese3,mese4,mese5} (and references therein).
		
		Finally, we mention the fact that considering $X=\Sigma_N$ as a target is equivalent to consider $X$ as a locally finite metric tree, which is indeed an example of NPC space. In this field, we mention the work by Sun \cite{sun} and references therein.
		
		% Let us now assume that $\mathcal{M}$ is the euclidean space $\mathcal{M}=\R^d$ with the standard metric. 

		% One can easily observe that, in the case the domain $M_1$ is an euclidean open set $D$ of $\R^d$ and the target $M_2$ is the set $\Sigma_N$, the energy functional becomes exactly $E(\cdot,D)$ defined in the introduction, acting on $H^1(D,\Sigma_N)$. Therefore, one can interpret minimizers of $E(\cdot,D)$ as a generalization of the classical notion of harmonic map, in the case in which the target space displays some kinds of singularities. For what concerns regularity issues, the study of harmonic maps with values  into singular spaces was initiated by Gromov and Schoen in \cite{gromov-schoen}. In particular, a property which allows for more structure in the problem is the curvature of the target space. As for the case of smooth target manifold, if the sectional curvature of the space is non-positive, minimizers are Lipschitz continuous. In the case of non-smooth target, a suitable generalization of the notion of \emph{upper bound on the sectional curvature} was introduced by Alexandrov and, in particular, we have that $\Sigma_N$ has non-positive curvature in this generalized sense. \red{Dire che se la varietà target liscia è NPC ho ulteriori proprietà di regolarità. Dire che i CAT(0) sono limiti di varietà lisce NPC, trovare referenze a riguardo (Gromov?).}
		
		\subsection{Limit of competition-diffusion models}\label{subsec:comp-diff} 
		For an open $D\sub\R^d$, we now consider the functional
		\begin{equation*}
			E_\beta(u,D):=\sum_{i=1}^N\int_D |\nabla u_i|^2\dx+\frac{\beta}{2}\sum_{\substack{i,j=1 \\ i\neq j}}^N \int_D u_i^2 u_j^2\dx,
		\end{equation*}
		defined for $u=(u_1,\dots,u_N)\in H^1(D,\R^N)$, i.e. with not necessarily segregated densities, and let us consider a non-negative minimizer $u_\beta\in H^1(D,\R^N)$ of $E_\beta(\cdot,D)$, that is $u_i\geq 0$ in $D$ for all $i=1,\dots,N$ and
		\begin{equation*}
			E_\beta(u_\beta,D)\leq E_\beta (v,D)\quad\text{for all }v\in H^1(D,\R^N)~\text{such that }u_\beta-v\in H^1_0(D,\R^N).
		\end{equation*}
		Differently from the optimal partition model considered in \Cref{sec:intr}, here we deal with unconstrained minimizers, which then satisfy the following system
		\begin{equation*}
			-\Delta u_{\beta,i}+\beta u_{\beta,i}\sum_{j\neq i} u_{\beta,j}^2=0,\quad\text{in }D~\text{for all }i=1,\dots,N.
		\end{equation*}
		As a consequence, by a bootstrap argument, we have that $u_{\beta,i}\in C^\infty(D)$ and, by the maximum principle, we have that either $u_{\beta,i}\equiv 0$ in $D$ or $u_{\beta,i}>0$ in $D$, for all $i=1,\dots,N$.
		However, by looking at the functional, one can easily see that the term
		\begin{equation*}
			\frac{\beta}{2}\sum_{\substack{i,j=1 \\ i\neq j}}^N \int_D u_i^2 u_j^2\dx,
		\end{equation*}
		penalizes the regions where two different densities coexists i.e.
		\begin{equation*}
			\bigcup_{\substack{i,j=1 \\ i\neq j}}^N \Big(\{u_{\beta,i}>0\}\cap \{u_{\beta,j}>0\}\Big),
		\end{equation*}
		and the penalization grows as $\beta>0$ increases. One might expect a sort of convergence to the segregated model described in \Cref{sec:intr}; indeed, this fits into the framework of $\Gamma$-convergence and one can easily prove the following.
		% Then, for any $i\in\{1,\dots,N\}$, $u_{\beta,i}$ satisfies
		% \begin{equation*}
			% 	-\Delta u_{\beta,i}+\beta u_{\beta,i}\sum_{j\neq i} u_{\beta,j}^2=0\quad\text{in }D.
			% \end{equation*}
		% If we let $V_{\beta,i}:=\beta \sum_{j\neq i} u_{\beta,j}^2\geq 0$, we thus have 
		% \begin{equation*}
			% 	-\Delta u_{\beta,i}+ V_{\beta,i}\,u_{\beta,i}=0\quad\text{in }D,
			% \end{equation*}
		% that is $V_{\beta,i}$ is a repulsive potential: by looking at the functional
		% \begin{equation*}
			% 	u\mapsto \int_D(|\nabla u|^2+Vu^2)\dx,
			% \end{equation*}
		% with $V\geq 0$, we see that a minimizer $u$ tends to live near the region where $V$ is small. In other words, we see that the term 
		% \begin{equation*}
			% 	\frac{\beta}{2}\sum_{\substack{i,j=1 \\ i\neq j}}^N \int_D u_i^2 u_j^2\dx
			% \end{equation*}
		% penalizes the regions $\{u_i>0\}\cap \{u_j>0\}$ for $i\neq j$ and $\beta>0$ plays the role of competition parameter. In fact, we can prove the following.
		\begin{proposition}
			$E_\beta(\cdot,D)\overset{\Gamma}{\longrightarrow}E_\infty(\cdot,D)$ as $\beta\to+\infty$, where
			\begin{equation*}
				E_\infty(u,D):=\begin{cases}
					\displaystyle E(u,D) &\text{if }u\in H^1(D,\Sigma_N), \\
					\displaystyle+\infty &\text{otherwise}.
				\end{cases}
			\end{equation*}
		\end{proposition}
		A standard consequence of the $\Gamma$-convergence is that minimizers converge to minimizers. More precisely, if $u_\beta\in H^1(D,\R^N)$ is a minimizer for $E_\beta(\cdot,D)$ and if $u_\beta\weak u$ weakly in $H^1(D,\R^N)$ as $\beta\to+\infty$, then $u\in H^1(D,\Sigma_N)$ and $u$ is a minimizer for $E(\cdot,D)$.
		% \begin{lemma}
			% 	Let $g\in H^1(D,\Sigma_N)$ and, for any $\beta>0$, let $u_\beta\in H^1(D,\R^N)$ be a minimizer for $E_\beta(\cdot,D)$ such that $u_\beta-g\in H^1_0(D,\R^N)$. If $u_\beta\weak u$ weakly in $H^1(D,\R^N)$ as $\beta\to+\infty$, then $u\in H^1(D,\Sigma_N)$, $u-g\in H^1_0(D,\R^N)$ and $u$ is a minimizer for $E_\infty(\cdot,D)$. Up to a subsequence, the weak convergence can be recovered by the fact that
			% 	\begin{equation*}
				% 		E_\beta(u_\beta,D)\leq E_\beta(g,D)=\sum_{i=1}^N\int_D|\nabla g_i|^2\dx.
				% 	\end{equation*}
			% \end{lemma}
		% \begin{proof}
			% 	This is a standard consequence of the $\Gamma$-convergence. For any $v\in H^1(D,\Sigma_N)$ such that $v-g\in H^1_0(D,\R^N)$, let $\{v_\beta\}_\beta$ be its recovery sequence. Then
			% 	\begin{equation*}
				% 		E_\infty(u,D)\leq \liminf_{\beta\to+\infty}E_\beta(u_\beta,D)\leq \limsup_{\beta\to+\infty}E_\beta(u_\beta,D)\leq \limsup_{\beta\to+\infty}E_\beta(v_\beta,D)\leq E_\infty(v,D).
				% 	\end{equation*}
			% \end{proof}
		
		% Hence, harmonic maps $u\colon D\to \Sigma_N$ emerge as limit of minimizers in competition-diffusion models.
		%The same limit appears also for the non-variational model
		%\begin{equation*}
		%		-\Delta u_i+\beta u_i\sum_{j\neq i}u_j=0\quad\text{in }D
		%\end{equation*}
		%where $u_i\geq 0$. 
		% One can consider the more general (evoloutionary) model
		% \begin{equation*}
			% 	\partial_t u_i-\dive(A_i \nabla u_i)+f(u_i)+ u_i\sum_{j\neq i}\beta_{ij}u_j^2=0\quad\text{in }D.
			% \end{equation*}
		This point of view was adopted in a series of pioneering papers by Conti-Terracini-Verzini \cite{CVT1,CTV2,CTV3}. Afterwards, a vast literature flourished in this framework, studying various aspects of this \enquote{approximate} model, which has deep connections with other important topics including Bose-Einstein condensates and Schr\"odinger systems, see e.g. \cite{NTTV1,NTTV2,TV,soave,SZ}.

		\subsection{Shape optimization}\label{subsec:optimal}
		
		Minimizers -- or, more precisely \emph{almost-minimizers} -- of $E(\cdot,D)$ acting on $H^1(D,\Sigma_N)$ emerge in the setting of optimal partition of the domain $D$ into $N$ subdomains with the cost function being the sum of the first Dirichlet eigenvalue on each subdomain. Let us be more precise. We consider the family of $N$-partitions of $D$
		\begin{equation*}
			\mathcal{P}_N(D):=\big\{(\Omega_1,\dots,\Omega_N)\colon \Omega_i\sub D~\text{open},~\Omega_i\cap\Omega_j=\emptyset~\text{for }i\neq j\big\},
		\end{equation*}
		and we consider, for $p\in(0,+\infty]$, the following optimal partition problem for $D$ (studied by Helffer, Hoffmann-Ostenhof \& Terracini in \cite{HHOT10} and by Conti, Terracini \& Verzini in \cite{CTVFucick2005})
		\begin{equation}\label{eq:shape_opt_p}
			E_N^p(D):=\inf\left\{ \left(\sum_{i=1}^N\lambda_1(\Omega_i)^p\right)^{\frac{1}{p}}\colon (\Omega_1,\dots \Omega_N)\in\mathcal{P}_N(D)\right\},
		\end{equation}
		where $\lambda_1(\Omega_i)$ stands for the first eigenvalue of the Dirichlet-Laplacian on $\Omega_i$ and, in the limit case $p=+\infty$, the value
		\begin{equation*}
			\left(\sum_{i=1}^N\lambda_1(\Omega_i)^p\right)^{\frac{1}{p}},
		\end{equation*}
		is intended as
		\begin{equation*}
			\max_{i\in\{1,\dots,N\}}\lambda_1(\Omega_i).
		\end{equation*} 
		In words, problem \eqref{eq:shape_opt_p} aims to find an optimal partition $\{\Omega_i\}_{i=1,\dots,N}$ of $D$ with respect to the cost functional
		\begin{equation*}
			\mathcal{P}_N(D) \ni (\Omega_1,\dots,\Omega_N) \mapsto \left(\sum_{i=1}^N \lambda_1(\Omega_i)^p\right)^{\frac{1}{p}},
		\end{equation*}
		given by the $p$-norm of the $N$-vector of the first Dirichlet eigenvalues on $(\Omega_1,\dots,\Omega_N)$: this clearly falls into the category of \emph{shape optimization problems}. For instance, to mention one remarkable example, in the particular case of the $3$-partition, i.e. $N=3$, and $D=B_1$, we have the following conjecture.
		\begin{problem}[Bishop-Friedland-Hayman-Helffer-Hoffmann-Ostenhof-Terracini]\label{conj}
			For $N=3$, prove that the $Y$-configuration, given by
			\begin{equation*}
				(\Omega_1,\Omega_2,\Omega_3)=\Big((\R^{d-2}\times \mathcal{Y}_1)\cap B_1\,,\,(\R^{d-2}\times \mathcal{Y}_2)\cap B_1\,,\,(\R^{d-2}\times \mathcal{Y}_3)\cap B_1\Big)
			\end{equation*}
			where 
			\begin{equation*}
				\mathcal{Y}_i:=\Big\{(r\cos\theta,r\sin\theta)\colon r>0,~2\pi(i-1)/3<\theta<2\pi i/3\Big\},\quad\text{for }i=1,2,3,
			\end{equation*}
			is the unique minimizer of \eqref{eq:shape_opt_p} with $D=B_1$, for any $p\in(0,+\infty]$ (up to rotations).
		\end{problem}
		This conjecture was proved for $p=+\infty$ in dimension $d=3$ by Helffer, Hoffmann-Ostenhof and Terracini in \cite{HHOT10} and in any dimension in \cite{OV2} (see also \cite{ST}), to which we refer for further details on this topic.
		
		These shape optimization problems turn out to fall into a larger class of problems of the kind
		\begin{equation*}
			\inf\Big\{F(\Omega_1,\dots,\Omega_N)\colon (\Omega_1,\dots,\Omega_N)\in\mathcal{P}_N(D)\Big\},
		\end{equation*}
		with $F$ belonging to a certain class of functionals. In this generality, questions of existence of optimal shapes were addressed in \cite{BBH1998} (see also the survey \cite{buttazzo}).

		Let us now analyze the connection with the optimal partition problem introduced in \Cref{def:minimizer}, focusing on the case $p=1$ in \eqref{eq:shape_opt_p}, that is
		\begin{equation}\label{eq:shape_opt}
			E_N(D):=\inf\left\{ \sum_{i=1}^N\lambda_1(\Omega_i)\colon (\Omega_1,\dots \Omega_N)\in\mathcal{P}_N(D)\right\}.
		\end{equation}
		A possible strategy to tackle the shape optimization problem \eqref{eq:shape_opt} is to pass to a relaxed formulation, which we now describe. For a measurable $\omega\sub D$, we let
		\begin{equation*}
			\lambda_1(\omega):=\inf\left\{\frac{\int_\omega |\nabla u|^2\dx}{\int_\omega u^2\dx}\colon u\in H^1_0(D)\setminus\{0\}~\text{is such that }u=0~\text{a.e. in }D\setminus\omega\right\}.
		\end{equation*}
		We then consider the following larger family of $N$-partitions of $D$
		\begin{equation*}
			\mathcal{P}_N^*(D):=\{(\Omega_1,\dots,\Omega_N)\colon \Omega_i\sub D~\text{measurable},~\Omega_i\cap\Omega_j=\emptyset~\text{for }i\neq j\}
		\end{equation*}
		and the corresponding relaxed problem
		\begin{equation}\label{eq:E*}
			E_N^*(D):=\inf\left\{\sum_{i=1}^N \lambda_1(\Omega_i)\colon (\Omega_1,\dots,\Omega_N)\in\mathcal{P}_N^*(D)\right\}.
		\end{equation}
		Trivially, we have that
		\begin{equation*}
			E_N^*(D)\leq E_N(D).
		\end{equation*}
		At this point, we can derive an equivalent formulation in terms of Sobolev functions with values in $\Sigma_N$. More precisely, one can prove the following.
		\begin{lemma}
			If we let
			\begin{multline}\label{eq:E**}
				E_N^{**}(D)=\inf\Bigg\{ \sum_{i=1}^N\frac{\int_D |\nabla u_i|^2\dx}{\int_D u_i^2\dx}\colon u\in H^1(D,\Sigma_N)\cap H^1_0(D,\R^N), \\ u_i\not\equiv 0~\text{for any }i\in\{1,\dots,N\} \Bigg\},
			\end{multline}
			then we have $E_N^*(D)=E_N^{**}(D)$. Moreover, if $u\in H^1(D,\Sigma_N)\cap H^1_0(D,\R^N)$ achieves \eqref{eq:E**}, then $(\Omega_1,\dots,\Omega_N)\in\mathcal{P}_N^*(D)$ achieves \eqref{eq:E*}, where
			\begin{equation*}
				\Omega_i:=\{x\in D\colon u_i(x)>0\}.
			\end{equation*}
			Viceversa, if $(\Omega_1,\dots,\Omega_N)\in\mathcal{P}_N^*(D)$ achieves \eqref{eq:E*}, then $u=(u_1,\dots,u_N)\in H^1(D,\Sigma_N)\cap H^1_0(D,\R^N)$ achieves \eqref{eq:E**}, where $u_i$ denotes the first eigenfunction on $\Omega_i$.
		\end{lemma}

		The next step is to prove that if $u=(u_1,\dots,u_N)$ achieves $E_N^{**}(D)$, then $u_i$ is continuous, so that the corresponding partition is open and we have
		\begin{equation*}
			E_N(D)=E_N^*(D)=E_N^{**}(D).
		\end{equation*}
		This path have been successfully pursued in \cite{CTV2003,CTVFucick2005,CL2007,CL2008}. In particular, the following holds (see e.g. \cite{CTVFucick2005}).
		\begin{proposition}
			Let $u\in H^1(D,\Sigma_N)\cap H^1_0(D,\R^N)$ be a minimizer of \eqref{eq:E**}. Then $u$ is locally Lipschitz continuous.
		\end{proposition}
		
		Finally, we emphasize that minimizers of  \eqref{eq:E**} can be proved to be almost minimizers for the optimal partition model problem defined in \Cref{def:minimizer}. For the proof we refer e.g. to \cite[Proposition 6.8]{OV1}.
		\begin{lemma}
			Let $u\in H^1(D,\Sigma_N)\cap H^1_0(D,\R^N)$ be a function achieving \eqref{eq:E**} and let $D'\Subset D$. Then, there exists $C=C(D')>0$ such that 
			\begin{equation*}
				E(u,B_r(x_0))\leq (1+C\,r)E(v,B_r(x_0))
			\end{equation*}
			for all $v\in H^1(B_r(x_0),\Sigma_N)$ such that $u-v\in H^1_0(B_r(x_0),\R^N)$ and for all $x_0\in D$ and $r>0$ such that $B_r(x_0)\sub D'$. If $D$ is regular, the result holds for every $D'\sub D$ and $C>0$ depends only on $d,N,D$.
		\end{lemma}
		
		\subsection{Symmetric 2-valued harmonic functions}\label{subsec:2-val}
		We start by briefly introducing the notion of a $q$-valued function, for more details we refer to \cite{de-lellis-spadaro-Memoirs,almgren} and the references therein.
		For any integer $q\ge 2$, we denote by $\mathcal A_q(\R)$ the set of all {\it unordered} $q$-uples of $\R$ and we will use the notation $a=\{a_1,\dots,a_q\}$ for an element $a$ of $\mathcal A_q(\R)$.
		Let $D$ be an open subset of $\R^d$. A function $\varphi:D\to \mathcal A_q(\R)$ is called a \emph{$q$-valued function}. Given a $q$-valued function $\varphi:D\to \mathcal A_q(\R)$ we will say that $\varphi\in C^{0,\alpha}$ if there is $C>0$ such that
		$$\text{dist}(\varphi(x),\varphi(y))\le C|x-y|^\alpha\quad\text{for all}\quad x,y\in D,$$
		where $\text{dist}(a,b)$ is the Wasserstein distance between $a=\{a_1,\dots,a_q\}$ and $b=\{b_1,\dots,b_q\}$, i.e.
		$$\text{dist}(a,b)=\inf_{\mathcal P}\sum_{i=1}^q|a_i-b_{\mathcal P(i)}|,$$
		the infimum being taken over all permutations $\mathcal P$ of $\{1,\dots,q\}$. Given a $q$-valued function $\varphi:D\to\mathcal A_q(\R)$, we say that $\varphi$ is differentiable in $D$, if for each $x\in D$ there is a $q$-valued affine function 
		$$L_x:\R^d\to\mathcal A_q(\R^d)\ ,\quad L_x(h)=\Big\{\varphi_1(x)+v_1(x)\cdot h\,,\dots,\,\varphi_q(x)+v_q(x)\cdot h\Big\},$$
		with $v_j(x)\in \R^d$ (single-valued), such that 
		$$\lim_{h\to0}\frac{1}{|h|}\text{dist}\left(\varphi(x),L_x(h)\right)=0.$$
		We will also write $\nabla \varphi:D\to\mathcal A_q(\R^d)$ to be the multivalued function 
		$$\nabla \varphi(x)=\{v_1(x),\dots,v_q(x)\}.$$
		Finally, we will say that say that $\varphi\in C^{1,\alpha}$, if $\varphi$ is differentiable and $\nabla \varphi\in C^{0,\alpha}$.  
		
		\vspace{0.2cm}
		
		\noindent We next recall the following definition of a harmonic $2$-valued function from \cite{simon-wickramasekera}. 
		
		\begin{definition}[Symmetric $2$-valued harmonic functions]\label{def:harmonic-2-valued}
			We will say that $\varphi:D\to\mathcal A_2(\R)$ is a \emph{symmetric $2$-valued map} if each $\varphi(x)$ is of the form $\{a,-a\}$ for some $a\in\R$. For a symmetric two-valued map $\varphi$, we denote by $\mathcal K_\varphi$ the set 
			$$\mathcal K_\varphi=\Big\{x\in D\ :\ \varphi(x)=\{0,0\}\ \text{and}\ \nabla \varphi(x)=\{0,0\}\Big\}.$$
			Moreover, we will say that a symmetric $2$-valued map $\varphi\in C^{1,\alpha}$ is \emph{harmonic} if for each ball $B\subset D\setminus\mathcal K_\varphi$ there is a harmonic function $h:B\to\R$ such that $\varphi\equiv\{h,-h\}$ in $B$.
		\end{definition}
		
		A typical example of a symmetric $2$-valued harmonic map is $$\varphi:\R^2\to\R\ ,\quad \varphi(z)=\Re\left(z^{\sfrac32}\right),$$
		for which $\mathcal K_u$ is the origin and the nodal set $\{\varphi=0\}$ divides the plane into three disjoint cones $\Omega_1,\Omega_2,\Omega_3$ with 120 degree angles at the origin. If we denote by $u_j$, $j=1,2,3$, the function 
		$$u_j(z)=\begin{cases}
			\vert\Re\Big(z^{\sfrac32}\Big)\vert,\ &\text{ if }\ z\in \Omega_j,\\
			0,\ &\text{if }\ z\notin \Omega_j,
		\end{cases}
		$$
		then the map 
		$$u:\R^2\to\Sigma_3\ ,\quad u=(u_1,u_2,u_3),$$
		is critical for the Dirichlet energy $E$ 
		%in the sense that $u\in\mathcal S(\R^2,3)$ 
		(in the sense of \cref{def:extremality} below). It is immediate to check that this is true in general. Indeed, we have the following proposition. 
		
		\begin{proposition}
			Suppose that $D$ is an open subset of $\R^d$ and that $\varphi:D\to\mathcal A_2(\R)$ is a $C^{1,\alpha}$ two-valued symmetric harmonic map with nodal set $\{\varphi=\{0,0\}\}$ that divides $D$ into $N$ distinct disjoint connected open sets $\Omega_1,\dots,\Omega_N$. 
			Let 
			$$u=(u_1,\dots,u_N):D\to\Sigma_N\ ,$$
			be the segregated map defined as follows:
			$$u_j(x)=\begin{cases}
				|\varphi(x)|\ ,&\text{if }\ x\in \Omega_j,\\
				0\ ,&\text{if }\ x\notin \Omega_j.
			\end{cases}$$
			Then, $u$ is critical ($u\in \mathcal S(D,N)$) is the sense of \cref{def:extremality}.
		\end{proposition}
		\begin{proof}
			It is sufficient to show that the components $u_i$, $i=1,\dots,N$, satisfy the inequalities 
			$$\Delta u_i\geq 0\quad\text{and}\quad 
			\Delta \Big(u_i-\sum_{j\neq i}u_j\Big)\leq 0\quad\text{in}\quad D.$$ 
			The first inequality is immediate since $u_i$ is non-negative in $D$ and harmonic in $\Omega_i=\{u_i>0\}$. Since, by \cite{simon-wickramasekera}, $\varphi\in H^2_{loc}(D)$ and the Hausdorff dimension of $\mathcal K_\varphi$ is at most $d-2$, it is sufficient to check that the second inequality holds in the interior of $D\setminus \mathcal K_\varphi$. Now, if $x_0$ lies in some $\Omega_k$, then $\Delta \Big(u_i-\sum_{j\neq i}u_j\Big)=0$. Suppose that $x_0$ lies on the nodal set $\{\varphi=0\}$. Since $x_0\notin\mathcal K_\varphi$, in a neighborhood of $x_0$ $\{\varphi=0\}$ is a smooth interface between two of the sets $\Omega_k$ and $\Omega_\ell$. But then, by a simple integration by parts we get that $\Delta \Big(u_i-\sum_{j\neq i}u_j\Big)\le 0$ in distributional sense around $x_0$. 
		\end{proof}
		
		We next show that also the converse is true, that is, every map in the class $\mathcal S$ (defined later in \Cref{def:extremality}) gives rise to a symmetric two-valued harmonic function. 
		
		\begin{proposition}
			Let $N\in\N$, $D\in\R^d$ be an open set, and let $u=(u_1,\dots,u_N)\in \mathcal S(D,N)$. Let $\Omega_j:=\{u_j>0\}$ for $j=1,\dots,N$. Let $\varphi:D\to\mathcal A_2(\R)$ be the symmetric two-valued function defined as
			$$\varphi(x):=\begin{cases}
				\{u_j(x),-u_j(x)\},&\text{if}\quad x\in\Omega_j\quad\text{for some}\quad j=1,\dots,N;\\
				\{0,0\},&\text{if}\quad x\in D\setminus \left(\bigcup_{j=1}^N\Omega_j\right).
			\end{cases}$$
			Then, $\varphi\in C^{1,1/2}$ and $\varphi$ is harmonic in the sense of \cref{def:harmonic-2-valued}.
		\end{proposition}
		\begin{proof}
			We first notice that the set $\mathcal K_\varphi$ coincides with the singular set $\mathcal S_D(u)$ defined in \cref{sec:interior}. Thus, for every $x_0\in \mathcal K_\varphi$ we have
			\begin{equation}\label{e:first-32-estimate-equivalence}
				|\varphi(x)|=|u|(x)=O(|x-x_0|^{3/2}),\quad\text{as }x\to x_0.
			\end{equation}
			Now, suppose that $B$ is a ball in $D\setminus \mathcal K_\varphi$. Since the nodal set of $u$ is smooth in $B$, thanks to \cite[Lemma 3.3]{OV2}, we can assign to each $u_j$ a sign $\sigma_j\in\{\pm 1\}$ in such a way that the resulting function $h(x)=\sum_{j=1}^N\sigma_ju_j(x)$ is harmonic in $B$ (see \cite[Lemma 3.4]{OV2}). This proves that $\varphi$ is harmonic in the sense of \cref{def:harmonic-2-valued}. In order to prove the $C^{1,1/2}$ regularity of $\varphi$, it is sufficient to prove that for each $x_0\in\mathcal K_\varphi$ we have
			$$|\nabla\varphi|(x)=O(|x-x_0|^{1/2}),\quad\text{as }x\to x_0.$$ 
			Let $x\in D\setminus \mathcal K_\varphi$, let $y$ be the projection of $x$ on $\mathcal K_\varphi$, and let $r=|x-y|$. Let $h$ be a harmonic function in $B_r(x)$ such that $\varphi=\{h,- h\}$. By \eqref{e:first-32-estimate-equivalence} and the gradient estimate we have that 
			$$|\nabla\varphi|(x)=|\nabla h|(x)\le C|x-y|^{1/2}\le C|x-x_0|^{1/2},$$
			where $C$ does not depend on the projection $y$, but only on the distance from $x_0$ to $\partial D$. 
		\end{proof}

		\section{Known results}\label{sec:known_results}
		
		In this section we consider minimizers (and a certain notion of critical points) of the Dirichlet energy, i.e. solutions of the optimal partition problem. We introduce some notation and the main tools needed in the investigation and we recall all the main results (to the best of our knowledge) concerning the regularity of the solution itself and of the free boundary.
		
		\subsection{Preliminaries}
		\subsubsection*{The classes $\mathcal{M}$ and $\mathcal{S}$}
		
		First of all, for the sake of clarity in the exposition, we recall here the notion of local minimizer, already introduced in \Cref{def:minimizer}. 
		%Given the set of admissible functions
		%  \begin{equation*}
			%     \mathcal{A}(D,N):=\{u\in H^1_{\textup{loc}}(D,\Sigma_N)\colon u_i\geq 0~\text{in }D\},
			% \end{equation*}
		% and the Dirichlet energy
		% \begin{equation*}
			%     E(u,D):=\int_D |\nabla u|^2\dx=\sum_{i=1}^N \int_D|\nabla u_i|^2\dx,
			% \end{equation*}
		% we recall the following.
		\begin{definition}\label{def:minimizer_1}
			Given $N\in\N$, an open set $D\subset\R^d$ and a function $u:D\to\R^N$, we will say that $u$ is admissible, and we will write $u\in\mathcal A(D,N)$, if $u\in H^1_{\textup{loc}}(D,\Sigma_N)$. We will say that an admissible function $u\in\mathcal A(D,N)$ is a \emph{local minimizer}, and we write $u\in\mathcal{M}(D,N)$, if for every open $\Omega\Subset D$ it holds
			\begin{equation*}
				E(u,\Omega)\leq E(v,\Omega)\quad\text{for all }v\in H^1(\Omega,\Sigma_N)~\text{such that }u-v\in H^1_0(\Omega,\R^N),
			\end{equation*}
			where $E(u,\Omega)$ is the Dirichlet energy
			\begin{equation*}
				E(u,\Omega):=\int_\Omega |\nabla u|^2\dx=\sum_{i=1}^N \int_\Omega|\nabla u_i|^2\dx.
			\end{equation*}
		\end{definition}
		
		We now introduce a notion of critical point of the Dirichlet energy acting on segregated Sobolev functions. 
		\begin{definition}\label{def:extremality}
			We say that $u\in \mathcal{A}(D,N)$ belongs to the class $\mathcal{S}(D,N)$ if
			\begin{equation}\label{eq:S}
				\begin{bvp}
					-\Delta u_i &\leq 0,&&\text{in }D, \\
					-\Delta \Big(u_i-\sum_{j\neq i}u_j\Big)&\geq 0, &&\text{in }D,
				\end{bvp}
			\end{equation}
			holds in the sense of distributions, for all $i=1,\dots,N$. These two inequalities are also known as \emph{extremality conditions}.
		\end{definition}
		This notion was first considered in a series of papers by Conti-Terracini-Verzini, see \Cref{subsec:comp-diff}, and we recall that there holds $\mathcal{M}(D,N)\sub\mathcal{S}(D,N)$ (see e.g. \cite[Theorem 5.1]{CTV3}).
		
		Finally, if $u\in \mathcal{M}(\R^d,N)$ (respectively, $u\in\mathcal{S}(\R^d,N)$) is homogeneous of some degree $\gamma>0$, we write $u\in \mathcal{M}_\gamma(\R^d,N)$ (respectively, $u\in\mathcal{S}_\gamma(\R^d,N)$).

		\subsubsection*{Lipschitz continuity}
		In \cite{CTV3} (see also \cite{gromov-schoen,CL2007} for the particular case of minimizers) it is proved that every $u\in \mathcal{S}(D,N)$ is locally Lipschitz continuous, i.e. $u_i\in C^{0,1}_{\textup{loc}}(D)$ for every $i=1,\dots,N$. In particular,  the sets
		\begin{equation*}
			\Omega^u_i:=\{x\in D\colon u_i(x)>0\}
		\end{equation*}
		are open and well-defined and, as a consequence of \eqref{eq:S}, the functions $u_i$ are harmonic in $\Omega_i^u$.
		
		\subsection{Interior regularity}\label{sec:interior}
		
		In the present section, we review the known results concerning the regularity of the free boundary at interior points.

		\subsubsection{Almgren and Weiss monotonicity formulas}\label{subsec:almgren}
		For any non-trivial $u\in\mathcal{S}(D,N)$ one can define the Almgren \emph{frequency function} for any $x_0\in D$ and any $r<\dist(x_0,\partial D)$ as
		\begin{equation*}
			\mathcal{N}(u,x_0,r):=\frac{E(u,x_0,r)}{H(u,x_0,r)},
		\end{equation*}
		where
		\begin{equation*}
			E(u,x_0,r):=\frac{1}{r^{d-2}}E(u,B_r(x_0))=\frac{1}{r^{d-2}}\sum_{i=1}^N\int_{B_r(x_0)}|\nabla u_i|^2\dx
		\end{equation*}
		is the scaled energy and
		\begin{equation*}
			H(u,x_0,r):=\frac{1}{r^{d-1}}\sum_{i=1}^N\int_{\partial B_r(x_0)}u_i^2\ds
		\end{equation*}
		is the scaled height function. It is known that for $u\in\mathcal{S}(D,N)$ which is not identically zero,  the function $r\mapsto \mathcal{N}(u,x_0,r)$
		is monotone non-decreasing and so the \emph{frequency} of $u$ at $x_0$
		\begin{equation*}
			\gamma(u,x_0):=\lim_{r\to 0^+}\mathcal{N}(u,x_0,r)
		\end{equation*}
		is well-defined at any point $x_0\in \mathcal F_D(u)$. Moreover, by standard arguments, the map $x_0\mapsto \gamma(x_0,u)$ is upper-semicontinuous. Finally,  once the expression of $\partial_r\mathcal{N}(u,x_0,r)$ is known, one can easily prove that also the Weiss energy
		\begin{equation*}
			W_\gamma(u,x_0,r):=\frac{H(u,x_0,r)}{r^{2\gamma}}\Big(\mathcal{N}(u,x_0,r)-\gamma\Big)
		\end{equation*}
		is monotone non-decreasing with respect to $r>0$, for any $\gamma\leq \gamma(u,x_0)$.
		
		\vspace{0.2cm}
		
		We remark that the frequency function $\mathcal{N}(u,x_0,r)$ could actually be defined only if $$H(u,x_0,r)>0,$$
		hence the monotonicity of $\mathcal{N}$ is conditional. On the other hand, with classical arguments, one can prove that if $u\in\mathcal{S}(D,N)$ is nontrivial, then $H(u,x_0,r)>0$ for every $x_0$ and $r>0$ such that $B_r(x_0)\sub D$ (see e.g. \cite[Theorem 2.2]{TT}). The proof of the Almgren monotonicity formula can be found, for instance, in \cite{TT} (see \cite{gromov-schoen} for the case of minimizers), while for the Weiss we can refer to \cite{OV1}.
		
		\subsubsection{Almgren blow-up}\label{sub:intro-blow-up} Let $u\in\mathcal{S}(D,N)\setminus\{0\}$, respectively $\mathcal{M}(D,N)\setminus\{0\}$, and let $x_0\in\mathcal F_D(u)$. We define the \emph{Almgren rescalings} as
		\[
		u_{x_0,r}(x):=\frac{u(x_0+rx)}{\sqrt{H(u,x_0,r)}}.
		\]
		with $H$ as in \Cref{subsec:almgren}, so that $H(u_{x_0,r},0,1)=1$. Furthermore, from scaling properties and the monotonicity of the Almgren frequency (see \Cref{subsec:almgren}), we have that
		\begin{equation*}
			\sum_{i=1}^N\int_{B_1}|\nabla (u_{x_0,r})_i|^2\dx=\mathcal{N}(u_{x_0,r},0,1)=\mathcal{N}(u,x_0,r)\leq \mathcal{N}(u,x_0,\dist(x_0,\partial D))
		\end{equation*}
		Hence, for any sequence $r_n\to0$ there is a subsequence $r_{n_k}\to 0$ such that
		\begin{equation*}
			u_{x_0,r_{n_k}}\to U\quad\text{uniformly in $B_1$ and in }H^1(B_1;\R^N),~\text{as }k\to\infty,
		\end{equation*}
		for some $U\in \mathcal{S}_\gamma(\R^d,N)$ ($U\in\mathcal{M}_\gamma(\R^d,N)$, respectively), with $\gamma=\lim_{r\to 0}\mathcal{N}(u,x_0,r)$, such that
		\begin{equation*}
			\sum_{i=1}^N\int_{\partial B_1}|U_i|^2\ds=H(U,0,1)=1.
		\end{equation*}
		Actually, from the monotonicity of the Almgren frequency, one immediately deduces only weak-$H^1$ convergence, but the full result only needs some further technical steps. For the complete proof we refer e.g. to \cite[Proposition 6.13]{OV1} (the arguments originally come from \cite{FFT2012}).
		
		\subsubsection{On the admissible frequencies}\label{subsec:classification}
		
		A crucial point in many problems in regularity theory is to understand whether a certain value $\gamma> 0$ is an \emph{admissible frequency}. In our context of the optimal partition problem, this means to investigate if $\mathcal{S}_\gamma(\R^d,N)\setminus\{0\}\neq \emptyset$ (respectively, $\mathcal{M}_\gamma(\R^d,N)\setminus\{0\}\neq \emptyset$), for some $N\in\N$. Let us review what is known in this topic.
		\begin{itemize}
			\item In view of the Lipschitz continuity of the solutions, one can easily prove that
			\begin{equation*}
				\mathcal{S}_\gamma(\R^d,N)\setminus\{0\}=\emptyset\quad\text{for all }0<\gamma<1,
			\end{equation*}
			see e.g. \cite[Corollary 2.7]{TT}. Moreover, we have that
			$$\mathcal{S}_1(\R^d,N)=\mathcal{S}_1(\R^d,2)=\mathcal{M}_1(\R^d,N)=\mathcal{M}_1(\R^d,2)$$
			and that if $U\in \mathcal{S}_1(\R^d,2)$, then $U$ has the form
			\begin{equation*}
				U_1=x_d^-,\qquad U_2=x_d^+,
			\end{equation*}
			up to rotations, multiplication by a dimensional constant and relabeling of the components' numbering.
			\item In \cite[Lemma 4.2]{ST} it was proved that
			\begin{equation*}
				\mathcal{S}_\gamma(\R^d,N)\setminus\{0\}=\emptyset\quad\text{for all }1<\gamma<\frac{3}{2}.
			\end{equation*}
			Moreover, basing ourselves on similar arguments, in \cite[Proposition 3.5]{OV2} we proved that 
			\begin{equation*}
				\mathcal{S}_{\sfrac{3}{2}}(\R^d,N)=\mathcal{S}_{\sfrac{3}{2}}(\R^d,3)=\mathcal{M}_{\sfrac{3}{2}}(\R^d,N)=\mathcal{M}_{\sfrac{3}{2}}(\R^d,3)
			\end{equation*}
			and that if $U\in \mathcal{S}_{\sfrac{3}{2}}(\R^d,3)$, then $U$ has the form
			$U(x)=Y(x_{d-1},x_d)$ with $2$-dimensional profile 
			defined as
			\begin{equation*}\label{eq:def-Y}
				Y_i(r,\theta)=\begin{cases}
					r^{\frac{3}{2}}\abs{\sin\left(\frac{3}{2}\theta\right)},&\text{for }\frac{2\pi}{3}(i-1)\leq \theta\leq \frac{2\pi}{3}i, \\
					0, &\text{elsewhere},
				\end{cases}
			\end{equation*}
			for $i=1,2,3$, up to rotations, multiplication by a dimensional constant and relabeling of the components' numbering.
			\item In \cite[Lemma 6.5]{OV2} we proved that there exists $\varepsilon_d>0$ such that
			\begin{equation*}
				\mathcal{M}_\gamma(\R^d,N)\setminus\{0\}=\emptyset\quad\text{for all }\frac{3}{2}<\gamma<\frac{3}{2}+\varepsilon_d.
			\end{equation*}
			
			\item If we restrict to dimension $d=2$, a full classification is available (for a deep analysis of the two dimensional case, we refer to \cite{CTV3} and \cite{HHOT09}).
			We know that
			\begin{equation}\label{eq:2D}
				\begin{gathered}
					\mathcal{S}_\gamma(\R^2,N)\setminus\{0\}\neq \emptyset\quad\text{if and only if } \\
					\gamma = \frac{m}{2}~\text{for some }m\in\N~\text{and }\begin{cases}
						2\leq N\leq m,&\text{if $N$ is even}, \\
						3\leq N\leq m,&\text{if $N$ is odd.}
					\end{cases}
				\end{gathered}
			\end{equation}
			In particular, this follows from the existence of the profile
			\begin{equation}\label{eq:2D_U}
				U_i(r,\theta):=\begin{cases}
					r^{\sfrac{m}{2}}\left|\sin\left(\frac{m}{2}\theta\right)\right|,&\text{for }\theta\in \left(\frac{2\pi}{m}(i-1),\frac{2\pi}{m}i\right), \\
					0,&\text{elsewhere},
				\end{cases}
			\end{equation}
			for $i=1,\dots,m$. By working on the labeling of the components of $U=(U_1,\dots,U_m)$, keeping in mind that two distinct connected components of $\{U_i>0\}$  can meet only at the origin, and thus cannot share a common boundary ray, one can easily produce minimal configurations for $N\in\N$ admissible as in \eqref{eq:2D}. Moreover, by exploiting the fact that we are in $2D$, one can also prove that if $U\in\mathcal{M}_{\gamma}(\R^d,N)\setminus\{0\}$ then $\gamma$ and $N$ must satisfy \eqref{eq:2D} and $U$ must be of the form \eqref{eq:2D_U} (up to relabeling of the vector's components). Furthermore, by a simple alternating argument (see e.g. \cite[Lemma 3.4]{OV2}), one can see that
			\begin{equation*}
				\mathcal{S}_{\sfrac{m}{2}}(\R^2,N)=\mathcal{M}_{\sfrac{m}{2}}(\R^2,N)
			\end{equation*}
			for all the admissible $N\in\N$.
			% for every number of domains $N\ge 2$ and for every integer $m\ge 2$, there is $U\in\mathcal S_{\sfrac{m}{2}}(\R^2,N)$ with the following properties: 
			% \begin{itemize}
				% \item up to rotations, $\mathcal{F}_{\R^2}(U)=\{(r,\theta)\colon r>0,~\theta=2\pi/m\}$;
				% \item for every $k\in\{1,\dots,N\}$, the set $\{U_k>0\}$ is empty or is the union of $n$ connected components of $\R^2\setminus\mathcal{F}_{\R^2}(U)$, with $n\in\{1,\dots,m\}$;
				% \item on each of these connected components, the function $U_k$ can be written in polar coordinates as
				% \[
				% U_k(r,\theta)=r^{\sfrac{m}{2}}\sin\left(\frac{m}{2}(\theta-\theta_0)\right),\quad \theta\in\left(\theta_0,\theta_0+\frac{2\pi}{m}\right),~r>0,
				% \]
				% up to multiplication by a dimensional constant (the same constant for all the components), where $\theta_0=2\pi i/m$, for some $i\in\{0,\dots,m-1\}$.
				% \end{itemize}
			%  First, we have that
			% \begin{equation*}
				%     \mathcal{S}_\gamma(\R^2,N)=\mathcal{M}_\gamma(\R^2,N)
				% \end{equation*}
			% for any $\gamma\geq 0$ and $N\in\N$.
			Finally, we observe that, by cylindrical extension in the remaining $d-2$ variables, there holds
			\begin{equation*}
				\emptyset\neq \mathcal{M}_{\sfrac{m}{2}}(\R^d,N)\setminus\{0\} \sub\mathcal{S}_{\sfrac{m}{2}}(\R^d,N)\setminus\{0\},
			\end{equation*}
			for any admissible $N\in\N$.
			\item Finally, spherical harmonics provide examples of minimizers in any dimension. Namely, if $p_m\colon \R^d\to \R$ is a nonzero harmonic polynomial, homogeneous of degree $m\in\N$ (whose restriction to the sphere $\mathbb{S}^{d-1}$ is called \emph{spherical harmonic}), then we can produce a nontrivial element of $\mathcal{M}_m(\R^d,N)$, for some $N\in\N$, as follows. It is known that $\R^d\setminus\{p_m=0\}$ is the union of a finite number $N$ of connected components $\{A_i\}_{i=1,\dots,N}$. Now, for any $i\in\{1,\dots,N\}$, we define 
			\begin{equation*}
				u_i:=\begin{cases}
					|p_m|,&\text{in }A_i, \\
					0,&\text{elsewhere}.
				\end{cases}
			\end{equation*}
			Finally, one can easily verify that $u=(u_1,\dots,u_N)\neq 0$ and that $u\in \mathcal{M}_m(\R^d,N)$.
		\end{itemize}
		
		\subsubsection{Decomposition of the free boundary}
		
		Given $\gamma\ge 1$, we use the following notation for the set of points of frequency $\gamma$     \begin{equation*}
			\mathcal{F}_D^\gamma(u):=\{x\in\mathcal{F}_D(u)\colon \gamma(u,x)=\gamma\}
		\end{equation*}
		(this is well defined in view of the validity of the Almgren's monotonicity formula).
		Hence, equivalently, in view of the blow-up analysis described in \Cref{subsec:almgren}, a value $\gamma\geq 1$ is admissible (in the sense of \Cref{subsec:classification}) if and only if $\mathcal{F}_D^\gamma(u)\neq \emptyset$ for some non-trivial $u\in\mathcal{S}(D,N)$ (respectively, $u\in\mathcal{M}(D,N)$), for some $N\in\N$. As a consequence of the results we just described in \Cref{subsec:classification}, we can split the free boundary as a disjoint union as follows:
		\begin{itemize}
			\item if $u\in\mathcal{S}(D,N)$, then
			%\begin{equation*}
			$\displaystyle    \mathcal{F}_D(u)=\mathcal{F}_D^1(u)\cup \Bigg[\bigcup_{\gamma\geq \sfrac{3}{2}}\mathcal{F}_D^\gamma(u)\Bigg];$
			%\end{equation*}
			\item if $u\in\mathcal{M}(D,N)$, then
			%\begin{equation*}
			$\displaystyle   \mathcal{F}_D(u)=\mathcal{F}_D^1(u)\cup\mathcal{F}_D^{\sfrac{3}{2}}(u)\cup \Bigg[\bigcup_{\gamma\geq \sfrac{3}{2}+\varepsilon_d}\mathcal{F}_D^\gamma(u)\Bigg].$
			%\end{equation*}
		\end{itemize}
		Finally, for reasons that will be clearer later, we also denote
		\begin{equation*}
			\mathcal{R}_D(u):=\mathcal{F}_D^1(u)\quad\text{and}\quad \mathcal{S}_D(u):=\bigcup_{\gamma\geq \sfrac{3}{2}}\mathcal{F}_D^\gamma(u).
		\end{equation*}
		
		\subsubsection{Clean-up}
		
		By a \emph{clean-up} result, we mean a statement of the type:
		\begin{center}
			\it if, in some ball, some components of a (normalized) solution $u\colon D\to \Sigma_N$ are small, \\
			then in a smaller ball these components vanish.
		\end{center}
		To be more precise, we here recall a tailored version of a more general result of Sun \cite{sun}, originally stated in the context of harmonic maps with values into $\R$-trees.
		\begin{theorem}[\cite{sun}, Theorem 1.2]
			Let $u\in\mathcal{M}(D,N)$, let $x_0\in D$ and let $R=\dist(x_0,\partial D)/2$. Then there exists $\delta,\rho>0$, depending on $d$ and $\mathcal{N}(u,x_0,R)$, such that the following holds. If, for some $\mathcal{I}\sub\{1,\dots,N\}$ and some $r\in(0,R)$, we have
			\begin{equation*}
				\sum_{i\in\mathcal{I}}\|(u_{x_0,r})_i\|_{L^\infty(B_1)}\leq \delta,
			\end{equation*}
			then
			\begin{equation*}
				(u_{x_0,r})_i\equiv 0\quad\text{in }B_\rho,~\text{for any }i\in\mathcal{I}.
			\end{equation*}
		\end{theorem}
		
		As a direct consequence, we also get the following.
		\begin{corollary}
			Let $u\in\mathcal{M}(D,N)$ and let $x_0\in D$. If, for some $r_n\to 0$ there holds
			\begin{equation*}
				u_{x_0,r_n}\to U\quad\text{uniformly in }B_1,
			\end{equation*}
			with $U\in\mathcal{M}_{\gamma(u,x_0)}(\R^d,N)$ satisfying $U_i\equiv 0$ for any $i\in\mathcal{I}$ (for some $\mathcal{I}\sub\{1,\dots,N\}$), then there exists $r>0$ such that $u_i\equiv 0$ in $B_r(x_0)$ for any $i\in\mathcal{I}$.
		\end{corollary}
		
		\subsubsection{Points of frequency 1} Let us consider $u\in\mathcal{S}(D,N)$ let $x_0\in \mathcal F^1_D(u)=\mathcal{R}_D(u)$. We recall that, in this case, all the limits of Almgren blow-up sequences are of the form
		\begin{equation*}
			U_1=x_d^-,\qquad U_2=x_d^+,
		\end{equation*}
		up to rotations, multiplication by a dimensional constant and relabeling of the components' numbering. We have that, around $x_0$, the free interface $\mathcal F_D(u)$ is a smooth manifold separating two of the sets $\Omega_i^u$. Precisely, there is a neighborhood $B_r(x_0)$ and two different indices $i,j\in \{1,\dots,N\}$ such that
		\[
		u_k\equiv 0\quad\text{in }B_r(x_0)\quad\text{for every}\quad k\notin\{i,j\},
		\]
		while in $B_r(x_0)$ the function $u_i-u_j$ is harmonic (thanks to the extremality conditions as in \Cref{def:extremality}) and 
		$$|\nabla (u_i-u_j)|(x_0)\neq 0.$$
		In particular, by the implicit function theorem, by choosing $r$ small enough, the interface 
		\begin{equation}\label{e:definition-Gamma-i-j}
			\partial\Omega_i^u\cap\partial\Omega_j^u\cap B_r(x_0)=\{x\in D\colon u_i(x)-u_j(x)=0\}\cap B_r(x_0),
		\end{equation}	
		is smooth $(d-1)$-manifold and  we have
		\[
		\mathcal F_D(u)\cap B_r(x_0)=\partial\Omega_i^u\cap B_r(x_0)=\partial\Omega_j^u\cap B_r(x_0).
		\]
		In particular, 
		\[
		\mathcal F_D(u)\cap B_r(x_0)=\mathcal R_D(u)\cap B_r(x_0).
		\]
		The proof of this fact can be found, for instance, in \cite{gromov-schoen}.
		\subsubsection{Points of frequency $\sfrac{3}{2}$}
		We now move to the study of the lowest stratum $\mathcal{F}_D^{\sfrac{3}{2}}(u)$ of the singular set  $\mathcal{S}_D(u)$. Before stating the results, we remark that the fine analysis of $\mathcal{F}_D^{\sfrac{3}{2}}(u)$ has been pursued, so far, only when $u\in\mathcal{M}(D,N)$ is a minimizer (this has been done in \cite{OV2}), while the corresponding analysis for the broader class  $\mathcal{S}(D,N)$ remains an open problem. Moreover, this is currently the only stratum of the singular set for which a full $\varepsilon$-regularity result is known.
		
		\vspace{0.2cm}
		
		\noindent We recall that, by the classification in \Cref{subsec:classification}, we know that any limit $U$ of Almgren blow-up sequences at points of frequency $\sfrac{3}{2}$ is of the form $U(x)=Y(x_{d-1},x_d)$ with $2$-dimensional profile 
		defined as
		\begin{equation}\label{eq:Y}
			Y_i(r,\theta)=\begin{cases}
				r^{\frac{3}{2}}\abs{\sin\left(\frac{3}{2}\theta\right)},&\text{for }\frac{2\pi}{3}(i-1)\leq \theta\leq \frac{2\pi}{3}i, \\
				0, &\text{elsewhere},
			\end{cases}
		\end{equation}
		for $i=1,2,3$, up to rotations, multiplication by a dimensional constant and relabeling of the components' numbering.
		
		\vspace{0.2cm}
		
		\noindent By upper-semicontinuity of the map $x_0\mapsto\gamma(u,x_0)$, we first notice that the set $\mathcal{F}_D^{\sfrac32}(u)$ is a relatively open subset of $\mathrm{S}_D(u)$. Moreover, in \cite{OV2} we proved, essentially, that near a point of frequency $\sfrac{3}{2}$ the solution is a $C^{1,\alpha}$ deformation of the model triple junction, see \Cref{fig:triple}. More precisely, for any $x_0\in\mathcal{F}_D^{\sfrac{3}{2}}(u)$ there exists $r>0$ and three different indices $i,j,k\in\{1,\dots,N\}$ such that
		\begin{itemize}
			\item $u_\ell\equiv 0$ in $B_r(x_0)$ for every $\ell\notin\{i,j,k\}$;
			\item $\mathcal{F}_D^{\sfrac{3}{2}}(u)\cap B_r(x_0)$ is a $(d-2)$-dimensional $C^{1,\alpha}$-smooth manifold;
			\item $\mathcal{F}_D(u)\cap B_r(x_0)$ is the union of three $(d-1)$-dimensional smooth manifolds (with boundary) $\Gamma_{ij}$, $\Gamma_{jk}$, $\Gamma_{ik}$, where $\Gamma_{h\ell}:=\partial\Omega_h^u\cap\partial\Omega_\ell^u\cap B_r(x_0)$, for $h,\ell\in\{i,j,k\}$;
			\item each pair of such manifolds intersects along $\mathcal{F}_{\sfrac{3}{2}}(u)\cap B_r(x_0)$
			\item $\Gamma_{h\ell}$ is $C^{1,\alpha}$-smooth up to $\mathcal F_{\sfrac32}(u)\cap B_r(x_0)$, for $h,\ell\in\{i,j,k\}$;
			\item the sets $\Gamma_{ij}$, $\Gamma_{jk}$, $\Gamma_{ik}$ form $120$ degree angles at $\mathcal F_{\sfrac32}(u)\cap B_r(x_0)$. 
		\end{itemize}
		
		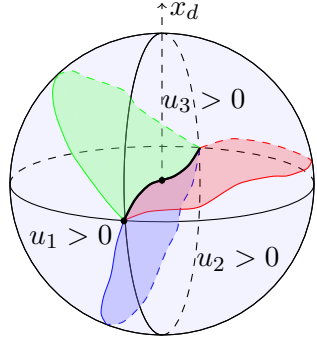
\begin{figure}[h]
			\centering
			\begin{tikzpicture}
				\coordinate (O) at (0,0);
				\filldraw (0,0.05) circle [radius = 1pt];
				\draw[fill=blue, opacity=0.05] (0,0) circle [radius = 20mm];
				\draw (0,0) circle [radius = 20mm];
				
				\draw[dashed, ->] (0,0) -- (0,2.4) node at (0.3, 2.3) {$x_d$};
				\draw[dashed] (0,-2) arc [start angle=-90, end angle = 90,x radius = 5mm, y radius = 20mm];
				\draw (0,2) arc [start angle=90, end angle = 270,x radius = 5mm, y radius = 20mm];
				\draw[dashed,  name path=red_arc_sopra] (-2,0) arc [start angle = 180, end angle = 76, x radius = 20mm, y radius = 5mm];
				\draw[name path=red_arc_sotto] (-2,0) arc [start angle=-180, end angle = -104,x radius = 20mm, y radius = 5mm];
				\draw [name path=g_arc_sotto] (2,0) arc [start angle=0, end angle = -104,x radius = 20mm, y radius = 5mm];
				\draw [dashed, name path=g_arc_sopra] (2,0) arc [start angle=0, end angle = 76,x radius = 20mm, y radius = 5mm];
				\draw [thick, name path=free_boundary] plot [smooth] coordinates {(-0.5,-0.48) (-0.35,-0.2) (-0.15,-0.0) (0.15,0.1) (0.35,0.25) (0.5,0.48)};

				%		\begin{scope}[transparency group,opacity=0.8]
					\filldraw (-0.5,-0.49) circle [radius = 1pt];
					%	    \draw [name path=p-blue, color=blue] plot [smooth] coordinates {(-0.5,-0.48) (-0.7,-1) (-0.85,-1.5) (-0.77,-1.79) (-0.74,-1.85) (-0.73,-1.86)};
					%	\filldraw (-0.73,-1.85) circle [radius = 1pt];
					\draw [name path=p-blue, color=blue] plot [smooth] coordinates {(-0.5,-0.48) (-0.6,-0.8) (-0.67,-1.2) (-0.77,-1.5)  (-0.8,-1.72) (-0.73,-1.86)};
					% \draw [name path=p-blue-dashed, color=blue, dashed] plot [smooth] coordinates {(0.5,0.48) (0.4,0) (0.05,-0.9) (-0.3,-1.5)  (-0.56,-1.76) (-0.73,-1.86)};
					\draw [name path=p-blue-dashed, color=blue, dashed] plot [smooth] coordinates {(0.5,0.48) (0.4,0) (0.25,-0.3) (0.2,-0.6)  (0.05,-0.9) (-0.2,-1.5)  (-0.5,-1.82) (-0.73,-1.86)};
					%   \filldraw (-0.6,-1.76) circle [radius = 1pt];
					%		\end{scope}
				%		
				%			\begin{scope}[transparency group,opacity=0.8]
					\filldraw (-0.5,-0.48) circle [radius = 1pt];
					\draw [name path=p-green, color=green] plot [smooth] coordinates {(-0.5,-0.48) (-0.75,-0.2) (-1.18,0.4) (-1.39,0.9) (-1.44,1.2)  (-1.41,1.37) (-1.37,1.43)};
					\draw [name path=p-green-dashed, color=green, dashed] plot [smooth] coordinates {(0.5,0.48) (0.3,0.6) (-0.3,1.0) (-0.5,1.24) (-1,1.42) (-1.18,1.5) (-1.35,1.46) (-1.37,1.43)};
					%		\end{scope}
				%		
				%			\begin{scope}[transparency group,opacity=0.8]
					
					\draw [name path=p-red, color=red] plot [smooth] coordinates {(-0.5,-0.48) (-0.2,-0.37) (0.1,-0.34) (0.5,-0.3) (1,0) (1.435,0.1) (1.65,0.11) (1.9,0.2) (1.969,0.3)};
					\draw [name path=p-red-dashed ,color=red, dashed]  plot [smooth] coordinates {(0.5,0.48) (0.9,0.6) (1.3,0.56) (1.6,0.48) (1.9,0.36) (1.969,0.3)};
					%	\end{scope}
				\draw (0,0) circle [radius = 20mm];
				\draw (0,2) arc [start angle=90, end angle = 270,x radius = 5mm, y radius = 20mm];
				\begin{scope}[transparency group,opacity=0.2]
					\tikzfillbetween[of=p-blue and p-blue-dashed] {color=blue!50};
					%	\tikzfillbetween[of=p-red and p-red-dashed] {color=red};	
					\tikzfillbetween[of=free_boundary and p-green] {color=green!50};
					\tikzfillbetween[of=free_boundary and p-green-dashed] {color=green!50};
					\tikzfillbetween[of=free_boundary and p-blue-dashed] {color=blue!50};
					%						\tikzfillbetween[of=p-red and free_boundary] {color=red};	
				\end{scope}

				\begin{scope}[transparency group,opacity=0.2]
					\tikzfillbetween[of=p-blue and p-blue-dashed] {color=blue!50};
					\tikzfillbetween[of=p-red and p-red-dashed] {color=red};	
					\tikzfillbetween[of=free_boundary and p-green] {color=green!50};
					\tikzfillbetween[of=free_boundary and p-green-dashed] {color=green!50};
					\tikzfillbetween[of=free_boundary and p-blue-dashed] {color=blue!50};
					\tikzfillbetween[of=p-red and free_boundary] {color=red};	
				\end{scope}
				\filldraw (-0.5,-0.49) circle [radius = 1pt];
				\draw [thick, name path=free_boundary] plot [smooth] coordinates {(-0.5,-0.48) (-0.35,-0.2) (-0.15,-0.0) (0.15,0.1) (0.35,0.25) (0.5,0.48)};
				\filldraw (0,0.05) circle [radius = 1pt];
				\draw node at (-1.2,-0.7) {$u_1>0$};
				\draw node at (1.0,-1.0) {$u_2>0$};
				\draw node at (0.55,1.1) {$u_3>0$};
			\end{tikzpicture}	
			\caption{A picture of a possible minimizer near a point of frequency $\sfrac{3}{2}$.}
			\label{fig:triple}
		\end{figure}
		
		\vspace{0.2cm}
		
		\noindent It is worth mentioning that a crucial role in the proof of these results is the epiperimetric inequality for the $\sfrac{3}{2}$ Weiss energy, again proved in \cite{OV2}, which we here recall.
		\begin{theorem}[Epiperimetric inequality]\label{t:epi}
			There exists $\delta,\eps,\tau\in (0,1)$ depending only on the dimension $d$ such that the following holds. For any $N\in\N$ and any $\sfrac{3}{2}$-homogeneous function $c\in H^1(B_1;\Sigma_N)\cap C^{0,1}(B_1;\R^N)$,
			%such that $c_{i}\geq 0$ for all $i=1,\dots,N$,
			\begin{equation*}
				\sum_{i=1}^3\d_{\mathcal{H}}\left(\{c_i>0\},\{Y_i>0\}\right)+\sum_{i=4}^N\d_{\mathcal{H}}\left(\{c_i>0\},\{x_{d-1}=x_d=0\}\right)\leq \tau
			\end{equation*}
			and
			\begin{equation*}
				\norm{\mathrm{d}_{\Sigma_N}(c,Y)}_{H^1(B_1)}\leq \delta,
			\end{equation*}
			with $Y$ as in \eqref{eq:Y} and $\d_{\Sigma_N}$ as in \eqref{eq:d_Sigma}, there exists $u\in H^1(B_1;\Sigma_N)\cap C^{0,1}(B_1;\R^N)$ such that 
			%$u_i\ge 0$ for all $i$, $u=c$ on $\partial B_1$, and
			\begin{equation*}
				W_{\frac{3}{2}}(u,0,1)\leq (1-\eps)W_{\frac{3}{2}}(c,0,1).
			\end{equation*}
		\end{theorem}

		\subsubsection{Rectifiability and measure bounds of the singular set}
		
		Besides $\varepsilon$-regularity results near particular frequencies, one can investigate dimension and measure bounds for the singular set, for both minimizers and critical points. While the set of regular points is known to be a smooth $(d-1)$-dimensional manifold, we also know that the singular set is strictly smaller. More precisely, if $u\in\mathcal{S}(D,N)$ then
		\begin{equation}\label{eq:hausdorff}
			\dim_{\mathcal{H}}(\mathcal{S}_D(u))\leq d-2
		\end{equation}
		and this result is clearly sharp, keeping in mind $2$-dimensional minimizers (and their cylindrical extensions). For reference, \eqref{eq:hausdorff} has been proved in \cite[Theorem 6.4]{gromov-schoen} in the case of minimizers and in \cite[Theorem 1.1]{TT} for the class $\mathcal{S}(D,N)$.
		
		\vspace{0.2cm}
		
		Actually, something more can be said. In particular, in \cite{Alper} (see also \cite{Dees}, which treats only minimizers, in the broader context of harmonic maps into singular spaces more general than $\Sigma_N$), the authors proved the following. For any $u\in \mathcal{S}(D,N)$, we have:
		\begin{itemize}
			\item the singular set $\mathcal{S}_D(u)$ is, locally, countably $(d-2)$-rectifiable; namely, it can be covered by countably many $(d-2)$-dimensional $C^1$ manifolds, up to a set of $(d-2)$-dimensional Hausdorff measure zero;
			\item for any compact $K\Subset D$ there exists $C>0$ and $r_0\in(0,\dist(K,\partial D)$ (depending on $d$, $u$ and $K$) such that
			\begin{equation*}
				\left|B_r(\mathcal{S}_D(u)\cap K)\right|\leq Cr^2\quad\text{for all }r\in(0,r_0),
			\end{equation*}
			where
			\begin{equation*}
				B_r(\mathcal{S}_D(u)\cap K):=\bigcup_{x_0\in\mathcal{S}_D(u)\cap K}B_r(x_0).
			\end{equation*}
			\item as a consequence of the previous point, $\mathcal{S}_D(u)$ has locally finite upper $(d-2)$-dimensional Minkowski content and $(d-2)$-dimensional Hausdorff measure.
		\end{itemize}
		We point out that the results contained in \cite{Alper,Dees} are based on the techniques introduced by Naber and Valtorta in the context of harmonic maps, see \cite{naber-valtorta}.

		\subsection{Boundary regularity}\label{sec:boundary}
		
		We now pass to the description of the known results for what concerns the regularity of both solutions and their free boundaries up to the fixed boundary $\partial D$. Compared to the interior case, to the best of our knowledge, there are few results concerning boundary regularity and they are essentially contained in \cite{serbinowski}, \cite{CTV3} and \cite{OV1}. In particular, in \cite{serbinowski}, the author works in the context of harmonic maps with values into NPC spaces and proves up-to-the-boundary H\"older regularity of minimizers, in terms of the H\"older character of the boundary itself and the Dirichlet datum. On the other hand, in \cite[Theorem 8.4]{CTV3}, the authors prove that if the domain $D$ is of class $C^1$ and if $u\in\mathcal{S}(D,N)$ satisfies $u=\phi$ on $\partial D$ for some $\phi\in W^{1,\infty}(\partial\Omega)$, then $u$ is locally Lipschitz continuous up to $\partial D$. Finally, in \cite{OV1} we focused on the regularity and the behavior of the free boundary $\mathcal{F}_D(u)$ up to the fixed boundary $\partial D$, when $u$ is a minimizer, i.e. $u\in\mathcal{M}(D,N)$, and satisfies homogeneous Dirichlet boundary conditions on a portion of $\partial D$. In particular, we analyze the regularity and geometrical structure of the \emph{free boundary on the fixed boundary}, defined as
		\begin{equation*}
			\mathcal{F}_{\partial D}(u):=\overline{\mathcal{F}_D(u)}\cap \partial D.
		\end{equation*}
		We now briefly discuss the main results of \cite{OV1} and we refer to the paper for all the details. We point out that the work \cite{OV1} actually deals with the spectral optimal partition problem described in \Cref{subsec:optimal}, which is slightly different from the problem of minimization of the pure Dirichlet energy, since there are also lower order terms to handle. However, one can easily see that all the results obtained for the former trivially hold true for the latter (which is the setting we are considering in the present survey).
		
		\subsubsection{Set-up and assumptions on $\partial D$} We begin by describing the setting of of \cite{OV1}, which is essentially laid as an extension of \Cref{sec:interior} in order to \enquote{see} the boundary $\partial D$. The definition of minimizer is the same as in \Cref{def:minimizer_1}, but with $\overline{D}$ replacing $D$ and prescribing homogeneous Dirichlet boundary conditions on a relatively open set $\Gamma\sub\partial D$. Namely, we say that $u\in\mathcal{M}_0^\Gamma(D,N)$ if it belongs to the admissible set
		\begin{equation*}
			\mathcal{A}_0^\Gamma(D,N):=\big\{u\in H^1_{\textup{loc}}(\overline{D},\Sigma_N)\colon 
			%u_i\geq 0~\text{in }D~\text{and }
			u_i\in H^1_{0,\Gamma}(\Omega)~\text{for any open bounded }\Omega\sub D\ \text{and all}\ i\ge 1\big\},
		\end{equation*}
		where $H^1_{0,\Gamma}(\Omega)$ is the space of $H^1$ functions in $\Omega$ that vanish on $\Gamma$ in the following sense
		\begin{equation*}
			H^1_{0,\Gamma}(\Omega):=\overline{C_c^\infty(\overline{\Omega}\setminus \Gamma)}^{\,\norm{\cdot}_{H^1(\Omega)}},
		\end{equation*}
		and if it minimizes the Dirichlet energy, i.e.
		\begin{equation*}
			E(u,\Omega)\leq E(v,\Omega)\quad\text{for all }v\in H^1(\Omega,\Sigma_N)~\text{such that }u-v\in H^1_0(\Omega,\R^N),
		\end{equation*}
		for any open bounded $\Omega\sub D$. Analogously, we say that $u\in\mathcal{S}_0^\Gamma(D,N)$ if $u\in\mathcal{S}(D,N)\cap (H^1_{0,\Gamma}(\Omega))^N$ for any open bounded $\Omega\sub D$. Finally, we analogously extend the notation from the interior case also for the $\gamma$-homogeneous functions, i.e. $\mathcal{M}_{0,\gamma}(\mathcal{C},N):=\mathcal{M}_{0,\gamma}^{\partial\mathcal{C}}(\mathcal{C},N)$ and $\mathcal{S}_{0,\gamma}(\mathcal{C},N):=\mathcal{S}_{0,\gamma}^{\partial \mathcal{C}}(\mathcal{C},N)$, for any open cone $\mathcal{C}\sub\R^d$. We emphasize that these spaces contain functions that vanish on the whole boundary of the cone $\partial \mathcal{C}$.
		
		\vspace{0.2cm}
		
		\noindent We now proceed by fixing the requirements on $\Gamma$. Without loss of generality, we describe the situation near the origin, assuming that $0\in\Gamma$. Then, we assume that, locally near $0$, $\Gamma=\partial D$ and that $D$ is the epigraph of a $C^1$ function $\varphi$, i.e.
		\begin{align*}
			D\cap B_1&=\{(x',x_d)\in \R^d\colon x_d>\varphi(x')\}\cap B_1, \\
			\partial D\cap B_1= \Gamma\cap B_1&=\{(x',x_d)\in\R^d\colon x_d=\varphi(x')\}\cap B_1,
		\end{align*}
		where, up to a rotation of the coordinate axes, we can suppose $\varphi(0)=|\nabla\varphi(0)|=0$. 
		
		\vspace{0.3cm}
		
		\begin{center}
			\it In the rest of the section we will work in the ball $B_1$ where $\Gamma=\partial D$ and $u\equiv 0$ on $\partial D$.\\
			In order to keep the notation as simple as possible we drop the index $\Gamma$.
		\end{center}
		
		\vspace{0.3cm}
		
		\noindent Moreover, if we denote by $\sigma\in C([0,2])$ the modulus of continuity of $\nabla \varphi$:
		\begin{equation*}
			|\nabla\varphi(x')-\nabla\varphi(y')|\leq \sigma(|x'-y'|)\quad\text{for all }x',y'\in B_1\cap\{x_d=0\},
		\end{equation*}
		we will assume that there exists $\sigma_0\in C^1(0,2)$ such that
		\begin{equation*}
			(r^{-m_d}\sigma_0(r))'\leq 0,\qquad \int_0^2\frac{\sigma_0(r)}{r}\,\d r<\infty,\qquad \int_0^2\frac{1}{r\sigma_0(r)}\int_0^r\frac{\sigma(t)}{t}\,\d t\,\d r<\infty,
		\end{equation*}
		where $m_d>0$ is a fixed dimensional constant. This last assumption is here stated in its generality, but it is actually satisfied if $\Gamma$ is of class $C^{1,\alpha-\text{Dini}}$ with $\alpha>3$ (see \cite[Definition 2.2]{OV1} and the discussion  in \cite{OV1}) and, in particular, if $\Gamma$ is of class $C^{1,\alpha}$, for some $\alpha\in(0,1)$.
		
		\subsubsection{Differentiability of minimizers}\label{subsec:diff}
		
		Let $u\in\mathcal{M}_0(D,N)\setminus \{0\}$. While from \cite[Theorem 8.4]{CTV3} it was already known that $u_i$ is Lipschitz up to $\partial D$, for any $i\in\{1,\dots,N\}$, in \cite[Theorem 2.3]{OV1} we proved that actually $u$ is differentiable at $0$ and that exactly one of the following holds
		\begin{enumerate}
			\item there exists $i\in\{1,\dots,N\}$ such that
			\begin{itemize}
				\item $u_i(x)=a_1\, x_d^++o(|x|)$ as $|x|\to 0$, for some $a_1>0$;
				\item $u_j\equiv 0$ in a neighborhood of $0$, for all $j\neq i$;
			\end{itemize}
			\item there exists $i,j\in\{1,\dots,N\}$, with $i\neq j$, $\bm{e}=(e_1,\dots,e_{d-1},0)\in\mathbb{S}^{d-2}\times\{0\}$ and $a_2>0$ such that
			\begin{itemize}
				\item $u_i(x)=a_2\,(x\cdot\bm{e})^+ \,x_d^++o(|x|^2)$ as $|x|\to 0$;
				\item $u_j(x)=a_2\,(x\cdot\bm{e})^- \,x_d^++o(|x|^2)$ as $|x|\to 0$;
				\item $u_k\equiv 0$ in a neighborhood of $0$, for all $k\neq i,j$;
			\end{itemize}
			\item $u_i(x)=o(|x|^2)$ as $|x|\to 0$ for any $i\in\{1,\dots,N\}$ and there exists $i,j\in\{1,\dots,N\}$, with $i\neq j$, such that $u_i\not\equiv 0$ and $u_j\not\equiv 0$ in a neighborhood of $0$.
		\end{enumerate}
		By scanning the proofs in \cite{OV1}, one can track the dependence of the constants $a_1$ on the point and derive its continuity, thus implying that $\nabla u_i(x_0)$ is continuous with respect to $x_0\in\partial D$, for any $i$.
		
		\subsubsection{Almgren and Weiss monotonicity formulas}
		
		Also in the boundary case, one can derive a good notion of frequency function and Weiss function, which turn out to be almost monotone. In particular, basing ourselves on the techniques introduced in \cite{Adolfsson1997} with the scope of studying harmonic functions vanishing on non-convex boundaries, in \cite{OV1} we proved that there exists a $C^1$ diffeomorphism $\Psi\colon \R^d\to\R^d$ such that $\Psi(0)=0$ and $D\Psi(0)=I$, which induces a suitable perturbation of balls in such a way that the corresponding frequency function is almost monotone. More precisely, we have that for any non-trivial $u\in\mathcal{M}_0(D,N)$ (extended by zero outside $D$) one can define the (boundary) Almgren's \emph{frequency function} as
		\begin{equation*}
			\mathcal{N}(u,r):=\frac{E(u,r)}{H(u,r)},
		\end{equation*}
		where for any $r\in(0,1)$
		\begin{equation*}
			E(u,r):=\frac{1}{r^{d-2}}E(u,\Psi(B_r))=\frac{1}{r^{d-2}}\sum_{i=1}^N\int_{\Psi(B_r)}|\nabla u_i|^2\dx
		\end{equation*}
		is the scaled energy and
		\begin{equation*}
			H(u,r):=\frac{1}{r^{d-1}}\sum_{i=1}^N\int_{\partial \Psi(B_r)}u_i^2\ds
		\end{equation*}
		is the scaled height function. Then, as proved in \cite[Theorem 6.6]{OV1}, the function $r\mapsto \mathcal{N}(u,r)$ is almost monotone, in the sense that there exists $C>0$ and $f\in L^1(0,1)$ such that
		\begin{equation*}
			\Big(e^{C\int_0^r f(t)\,\d t}\mathcal{N}(u,r)\Big)'\geq 0.
		\end{equation*}
		Hence, we can define the \emph{frequency} of $u$ at $0$
		\begin{equation*}
			\gamma(u,0):=\lim_{r\to 0^+}\mathcal{N}(u,r).
		\end{equation*}
		Analogously, one can define the frequency function $\mathcal{N}(u,x_0,r)$ and its limit $\gamma(u,x_0)$ at any $x_0\in\partial D$. Again, by standard arguments, the map $x_0\mapsto \gamma(x_0,u)$ is upper-semicontinuous. Finally, we point out that also the Weiss energy
		\begin{equation*}
			W_\gamma(u,r):=\frac{H(u,r)}{r^{2\gamma}}\Big(\mathcal{N}(u,r)-\gamma\Big)
		\end{equation*}
		is almost monotone non-decreasing with respect to $r>0$, for any $\gamma\leq \gamma(u,0)$, see \cite[Proposition 6.16]{OV1}.
		
		\subsubsection{Almgren blow-up}
		
		With some technical adjustments, by virtue of the almost monotonicity of a perturbed frequency function (see the previous paragraph), we can perform a blow-up analysis in the same spirit as \cref{subsec:almgren}. More precisely, setting
		%in this case we have the following. Letting
		\begin{equation*}
			u_{x_0,r}(x):=\frac{u(x_0+rx)}{\sqrt{H(u,x_0,r)}}\,,
		\end{equation*}
		we have that, for any sequence $r_n\to 0$ there exists a subsequence $r_{n_k}\to 0$ such that
		\begin{equation*}
			u_{x_0,r_{n_k}}\to U\quad\text{uniformly in }B_1~\text{and in }H^1(B_1;\R^N),
		\end{equation*}
		as $k\to\infty$, for some $U\in\mathcal{M}_{0,\gamma}(\R^d_+,N)$ satisfying
		\begin{equation*}
			\sum_{i=1}^N\int_{\partial B_1}|U_i|^2\ds=H(U,0,1)=1,
		\end{equation*}
		where $\gamma=\lim_{r\to 0}\mathcal{N}(u,x_0,r)$ and $\R^d_+:=\{(x',x_d)\in\R^d\colon x_d>0\}$.
		
		\subsubsection{Admissible frequencies} Again, the description of the results in this section follows the structure of the interior case. In particular, we again call $\gamma\geq 0$ an admissible frequency if $\mathcal{M}_{0,\gamma}(\R^d_+,N)\setminus\{0\}\neq \emptyset$ for some $N\in\N$. We have the following results, which are mainly contained in \cite{OV1}.% (some facts are trivial).
		\begin{itemize}
			\item In view of the Lipschitz continuity of minimizers, we have that
			\begin{equation*}
				\mathcal{M}_{0,\gamma}(\R^d_+,N)\setminus\{0\}=\emptyset\quad\text{for all }\gamma<1.
			\end{equation*}
			Moreover, 
			\begin{equation*}
				\mathcal{M}_{0,1}(\R^d_+,N)=\mathcal{M}_{0,1}(\R^d_+,1)=\mathcal{S}_{0,1}(\R^d_+,1)=\mathcal{S}_{0,1}(\R^d_+,N)
			\end{equation*}
			and these classes only contain functions of the form
			\begin{equation*}
				U_1=x_d^+,\qquad U_i\equiv 0\quad\text{for }i\geq 2,
			\end{equation*}
			up to rotations, multiplication by a dimensional constant and relabeling of the components' numbering.
			\item Furthermore
			\begin{equation*}
				\mathcal{M}_{0,\gamma}(\R^d_+,N)\setminus\{0\}=\emptyset\quad\text{for all }1<\gamma<2
			\end{equation*}
			and 
			\begin{equation*}
				\mathcal{M}_{0,2}(\R^d_+,N)=\mathcal{M}_{0,2}(\R^d_+,2)
			\end{equation*}
			which contains only configurations of the form
			\begin{equation*}
				U_1=x_{d-1}^-\,x_d^+,\qquad U_2=x_{d-1}^+\,x_d^+,\qquad U_i\equiv 0\quad\text{for }i\geq 3,
			\end{equation*}
			again up to rotations, multiplication by a dimensional constant and relabeling of the components' numbering.
			\item There exists $\varepsilon_d'>0$ such that
			\begin{equation*}
				\mathcal{M}_{0,\gamma}(\R^d_+,N)\setminus\{0\}=\emptyset\quad\text{for all }2<\gamma<2+\varepsilon_d',
			\end{equation*}
			for any $N\in\N$.
			\item Reasoning analogously to the interior case, we can easily obtain the full classification in dimension $2$. More precisely, we know that
			\begin{equation}\label{eq:2D_b}
				\mathcal{S}_{0,\gamma}(\R^2_+,N)\setminus\{0\}\neq \emptyset\quad\text{if and only if }
				\gamma = m~\text{for some }m\in\N~\text{and }2\leq N\leq m.
			\end{equation}
			In particular, this follows from the existence of the profile
			\begin{equation}\label{eq:2D_U_b}
				U_i(r,\theta):=\begin{cases}
					r^{m}\left|\sin\left(m\theta\right)\right|,&\text{for }\theta\in \left(\frac{\pi}{m}(i-1),\frac{\pi}{m}i\right), \\
					0,&\text{elsewhere},
				\end{cases}
			\end{equation}
			for $i=1,\dots,m$. Again as in the interior case, by working on the labeling of the components of $U=(U_1,\dots,U_m)$, and keeping in mind that two distinct connected components of $\{U_i>0\}$ cannot share a line, one can easily produce minimal configurations for any $2\leq N\leq m$. Moreover, if $U\in\mathcal{M}_{0,\gamma}(\R^d_+,N)\setminus\{0\}$ then $\gamma$ and $N$ must satisfy \eqref{eq:2D_b} and $U$ must be of the form \eqref{eq:2D_U_b} (up to relabeling of the vector's components). Furthermore, one can also see that
			\begin{equation*}
				\mathcal{S}_{0,m}(\R^2_+,N)=\mathcal{M}_{0,m}(\R^2_+,N)
			\end{equation*}
			for all the admissible $N\in\N$. Finally, we observe that, by cylindrical extension in the remaining $d-2$ variables, there holds
			\begin{equation*}
				\emptyset\neq \mathcal{M}_{0,m}(\R^d_+,N)\setminus\{0\} \sub\mathcal{S}_{0,m}(\R^d_+,N)\setminus\{0\},
			\end{equation*}
			for any $2\leq N\leq m$.
			\item Finally, again with an argument analogous to the interior case, we can produce functions of $\mathcal{M}_{0,m}(\R^d_+,N)$ starting from harmonic homogeneous polynomials of $\R^d$ which are odd with respect to $\partial\R^d_+$ (hence vanishing on it).
		\end{itemize}
		
		\subsubsection{Traces of the positivity sets and decomposition of the free boundary}\label{subsec:trace}
		
		Given $\gamma\ge 1$, we use the following notation for the set of points of frequency $\gamma$     \begin{equation*}
			\mathcal{F}_{\partial D}^\gamma(u):=\{x\in\partial D\colon \gamma(u,x)=\gamma\}.
		\end{equation*}
		Analogously to the interior case, in view of the Almgren blow-up analysis, $\gamma\geq 1$ is admissible (in the sense of \Cref{subsec:classification}) if and only if $\mathcal{F}_{\partial D}^\gamma(u)\neq \emptyset$ for some non-trivial $u\in\mathcal{M}_0(D,N)$, for some $N\in\N$ (and for some $D\sub\R^d$). Moreover, we have the following correspondence with the results described in \Cref{subsec:diff} (we use the corresponding numbering):
		\begin{enumerate}
			\item[$-$] (1) happens if and only if $\gamma(u,0)=1$. In this case, we say that $0$ belongs to $\omega_i^u$. The set $\omega_i^u$ can be interpreted as the \emph{trace} of the positivity set 
			$$\Omega_i^u=\{x\in D\colon u_i(x)>0\}.$$ 
			Moreover, we can identify $\omega_i^u$ as follows:
			\begin{equation*}
				\omega_i^u=\{x\in\partial D\colon \partial_{\nnu}u_i(x)<0\}.
			\end{equation*}
			\item[$-$] (2) happens if and only if $\gamma(u,0)=2$. In this case, we say that 
			$$0\in \mathcal{R}_{\partial D}(u):=\mathcal{F}_{\partial D}^2(u);$$
			\item[$-$] (3) happens if and only if $\gamma(u,0)\geq 2+\varepsilon_d'$. In this case, we say that 
			$$0\in\mathcal{S}_{\partial D}(u):=\bigcup_{\gamma\geq 2+\varepsilon_d'}\mathcal{F}_{\partial D}^\gamma(u).$$
		\end{enumerate}
		
		\subsubsection{Regularity of the free boundary}
		
		The first observation is that 
		\begin{equation}
			\mathcal{F}_{\partial D}(u)=\bigcup_{i=1}^N \partial_{\partial D}\,\omega_i^u,
		\end{equation}
		where $\partial_{\partial D}\,\omega_i^u$ is the boundary of $\omega_i^u$ in the relative topology of $\partial D$.
		This is not trivial, since a priori only the inclusion $\supseteq$ holds, and needs a careful justification (see \cite[Proposition 9.4]{OV1}). In particular, this excludes behaviors like the one on the right in the figure below.\footnote{For aesthetic reasons, we omitted the upper index $u$ in $\Omega_i^u$ and $\omega_i^u$.}
		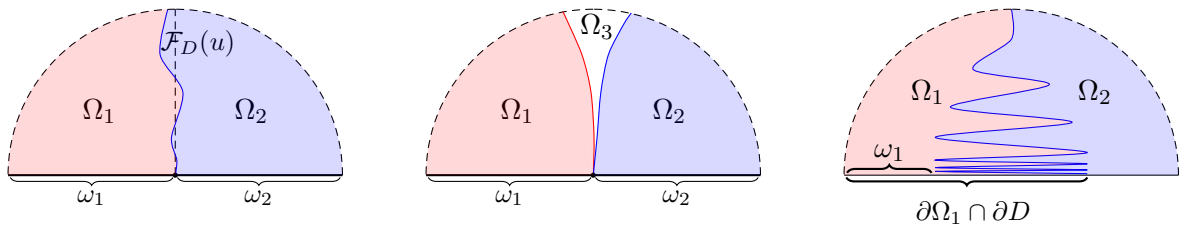
\begin{figure}[h]
			\centering
			\begin{tikzpicture}
				\draw[densely dashed, name path=g_arc_1] (2.2,0) arc [start angle=0, end angle = 90,x radius = 22mm, y radius = 22mm];
				\draw[densely dashed, name path=g_arc_2] (0,2.2) arc [start angle=90, end angle = 180,x radius = 22mm, y radius = 22mm];
				\draw [densely dashed, name path=blow] plot [smooth] coordinates {(0,0) (0,2.2)};
				\draw [blue, name path=vertical] plot [smooth] coordinates {(0,0) (0.02,0.2) (-0.05,0.5) (0.1,1.1) (-0.2,1.6) (-0.1,2.19)};
				%\draw [thick, red, dashed,  name path=vertical] plot [smooth] coordinates {(0,0) (0.02,0.2) (-0.05,0.5) (0.1,1.1) (-0.2,1.6) (-0.1,2.19)};
				\draw [thick, name path=hor] plot [smooth] coordinates {(-2.2,0) (2.2,0)};
				\draw node at (-1.0,0.85) {$\Omega_1$};
				\draw node at (1.0,0.85) {$\Omega_2$};
				\draw node at (0.3,1.75) {\small$\mathcal F_D(u)$};
				\draw [decorate,decoration={brace,mirror, amplitude=3pt}]
				(-2.17,-0.03) -- (-0.03,-0.03);
				\draw node at (-1.1,-0.3) {\small$\omega_1$};
				\draw [decorate,decoration={brace,mirror, amplitude=3pt}]
				(0.03,-0.03) -- (2.17,-0.03);
				\draw node at (1.1,-0.3) {\small$ \omega_2$};
				\draw[thick] (0,0) circle [very thick, radius=0.02cm];
				\begin{scope}[transparency group,opacity=0.15]
					\tikzfillbetween[of=g_arc_1 and hor]  {color=blue};
					\tikzfillbetween[of=g_arc_2 and hor]  {color=blue};
					\tikzfillbetween[of=g_arc_2 and vertical] {color=red};
				\end{scope}
				
				\begin{scope}[shift={(5.5,0)}]
					\draw[densely dashed, name path=g_arc_3] (2.2,0) arc [start angle=0, end angle = 90,x radius = 22mm, y radius = 22mm];
					\draw[densely dashed, name path=g_arc_4] (0,2.2) arc [start angle=90, end angle = 180,x radius = 22mm, y radius = 22mm];
					\draw[draw=none, name path=g_arc_5] (2.2,0) arc [start angle=0, end angle = 77,x radius = 22mm, y radius = 22mm];
					\draw[draw=none, name path=g_arc_6] (-2.2,0) arc [start angle=180, end angle = 100,x radius = 22mm, y radius = 22mm];
					%\draw [densely dashed, name path=blow] plot [smooth] coordinates {(0,0) (0,2.2)};
					\draw [red, name path=vertical1] plot [smooth] coordinates {(0,0) (0.01,0.5) (-0.03,1.1) (-0.15,1.6) (-0.4,2.17)};
					\draw [blue, name path=vertical2] plot [smooth] coordinates {(0,0) (0.02,0.2) (0.05,0.5) (0.1,1.1) (0.2,1.6) (0.4,2) (0.5,2.15)};
					\draw [thick, name path=hor2] plot [smooth] coordinates {(0,0) (2.2,0) };
					\draw [thick, name path=hor1] plot [smooth] coordinates {(0,0) (-2.2,0) };
					\draw node at (-1.0,0.85) {$\Omega_1$};
					\draw node at (1.0,0.85) {$\Omega_2$};
					\draw node at (0.05,1.97) {$\Omega_3$};
					\draw [decorate,decoration={brace,mirror, amplitude=3pt}]
					(-2.17,-0.03) -- (-0.03,-0.03);
					\draw node at (-1.1,-0.3) {\small$\omega_1$};
					\draw [decorate,decoration={brace,mirror, amplitude=3pt}]
					(0.03,-0.03) -- (2.17,-0.03);
					\draw node at (1.1,-0.3) {\small$ \omega_2$};
					\draw[thick] (0,0) circle [very thick, radius=0.02cm];
					\begin{scope}[transparency group,opacity=0.15]
						\tikzfillbetween[of=vertical1 and hor1]  {color=red};
						\tikzfillbetween[of=hor2 and vertical2]  {color=blue};
						\tikzfillbetween[of=g_arc_6 and hor1]  {color=red};
						\tikzfillbetween[of=g_arc_5 and hor2]  {color=blue};
						%\tikzfillbetween[of=g_arc_4 and vertical1] {color=red};
					\end{scope}
					
				\end{scope}

				\begin{scope}[shift={(11,0)}]
					\draw [blue, name path=osc] plot [smooth] coordinates {(0,2.2) (0,1.8) (-0.5,1.4) (0.5,1.2) (-0.8,0.9) (0.8,0.7) (-1,0.5) (1,0.3) (-1,0.2) (1,0.15) (-1,0.1) (1,0.07) (-1,0.05) (1,0.02)};
					%\draw [red, dashed] plot [smooth] coordinates {(0,2.2) (0,1.8) (-0.5,1.4) (0.5,1.2) (-0.8,0.9) (0.8,0.7) (-1,0.5) (1,0.3) (-1,0.2) (1,0.15) (-1,0.1) (1,0.07) (-1,0.05) (1,0.02)};
					\draw[densely dashed, name path=g_arc_1] (2.2,0) arc [start angle=0, end angle = 90,x radius = 22mm, y radius = 22mm];
					\draw[densely dashed,  name path=g_arc_2] (0,2.2) arc [start angle=90, end angle = 180,x radius = 22mm, y radius = 22mm];
					\draw [name path=hor] plot [smooth] coordinates {(-2.2,0) (2.2,0)};
					\begin{scope}[transparency group,opacity=0.15]
						\tikzfillbetween[of=g_arc_1 and hor] {color=blue};
						\tikzfillbetween[of=g_arc_2 and osc] {color=red};
					\end{scope}
					\draw node at (-1.1,1.1) {$\Omega_1$};
					\draw node at (1.1,1.1) {$\Omega_2$};
					\draw node at (-1.6,0.3) {$\omega_1$};
					%\draw node at (0.9,-0.2) {$\partial D$};
					%\draw[very thick,-{latex}] (0,0) -- (6,0) node[below]{$x$};
					%\foreach \x/\l in {1/-r,3/{},5/+r}
					%\draw (\x,3pt) -- (\x,-3pt) node[below]{$x_0\l$};
					\draw [thick, decorate,decoration={brace,amplitude=3pt}]
					(-2.16,0.03) -- (-1.05,0.03);
					\draw [thick, decorate,decoration={brace,mirror, amplitude=3pt}]
					(-2.17,-0.03) -- (1,-0.03);
					\draw node at (-0.5,-0.5) {\small$\partial \Omega_1\cap \partial D$};
				\end{scope}
				
			\end{tikzpicture}
			\caption{A regular free interface (on the left), a cusp-like singularity (in the middle), and an oscillating free boundary (on the right). We show that, among these, only the behavior on the left is possible.}
			\label{fig:three}
		\end{figure}

		\noindent At this point, we know that only (2) and (3) in \Cref{subsec:diff} and \Cref{subsec:trace} can hold if $0\in\mathcal{F}_{\partial D}(u)$ and that
		\begin{equation*}
			\mathcal{F}_{\partial D}(u)=\mathcal{R}_{\partial D}(u)\cup \mathcal{S}_{\partial D}(u).
		\end{equation*}
		We are now able to state the two main results of \cite{OV1}. The first one concerns the regularity of $\mathcal{R}_{\partial D}(u)$ from the \enquote{point of view} of $\partial D$.
		
		\begin{theorem}\label{thm:fixed_free_bound}
			Let $u\in\mathcal{M}_0(D,N)\setminus\{0\}$. We have that $\mathcal{R}_{\partial D}(u)$ and $\mathcal{S}_{\partial D}(u)$ satisfy the following: $\mathcal{S}_{\partial D}(u)$ is a relatively closed set and $\mathcal{R}_{\partial D}(u)$ is, locally, a $(d-2)$-dimensional submanifold of class $C^1$. Moreover,  for any $x_0\in\mathcal{R}_{\partial D}(u)$ there exist $i,j\in\{1,\dots,N\}$, with $i\neq j$, and $r_0>0$ such that 
			\begin{align*}
				&\omega_i^u\cap B_r(x_0)\neq \emptyset,\quad\omega_j^u\cap B_r(x_0)\neq \emptyset, \\
				&\omega_k^u\cap B_r(x_0)=\emptyset\ \text{ for all }\ k\neq i,j,
			\end{align*}
			for all $r\leq r_0$.
		\end{theorem}
		
		This, by itself, does not provide the complete picture, since there might be cases in which $\partial \omega_i^u\cap\partial \omega_j^u$ is regular near some $x_0\in\partial D$, but $x_0\not\in\mathcal R_{\partial D}(u)$. In the following, we describe the behavior of the regular part of interior free boundary as it approaches the fixed boundary $\partial D$. This also rules out the middle case in \Cref{fig:three}.
		
		\begin{theorem}\label{thm:up_to_the_bound}
			Let $u\in\mathcal{M}_0(D,N)\setminus\{0\}$. For any $x_0\in\mathcal{R}_{\partial D}(u)$, let $r_0>0$ be as in \Cref{thm:fixed_free_bound}. We have that $\overline{\mathcal{R}(u)}\cap B_{r_0}(x_0)$ is of class $C^1$ up to $\partial D$ and %there holds
			\begin{equation*}
				\mathcal{R}_{\partial D}(u)\cap B_{r_0}(x_0)=\overline{\mathcal{R}_D(u)}\cap \partial D\cap B_{r_0}(x_0).
			\end{equation*}
			Furthermore, in this case, $\mathcal{R}_D(u)$ approaches $\partial D$ in an orthogonal way, in the sense that, if $\bm{e}_x\in\partial B_1$ denotes a unit normal vector for $\mathcal{R}_D(u)\cap B_{r_0}(x_0)$ at the point $x\in \mathcal{R}_D(u)\cap B_{r_0}(x_0)$ and $\nnu(x_0)$ denotes the unit outer normal to $\partial D$ at $x_0$, then
			\[
			\lim_{\substack{x\to x_0 \\ x \in D}}\bm{e}_x\cdot\nnu(x_0)=0.
			\]
			
		\end{theorem}
		
		\vspace{1cm}

		\subsection*{Acknowledgements}
		The authors were supported by the European Research Council (ERC), through the European Union’s Horizon 2020 project ERC VAREG - Variational approach to the regularity of the free boundaries (grant agreement No. 853404). The authors acknowledge the MIUR Excellence Department Project awarded to the Department of Mathematics, University of Pisa, CUP I57G22000700001. B.V. also acknowledges support from the project MUR-PRIN “NO3” (n.2022R537CS).  
		
		\subsection*{Statements and Declarations} \emph{Competing Interests.} The authors have no competing interests to declare.

		\bibliographystyle{acm}
		\bibliography{biblio}

	\end{document}